\documentclass[11pt]{amsart}
\usepackage[margin=0.9in]{geometry}
\pdfoutput=1
\usepackage[backend=biber, backref, style=alphabetic]{biblatex}
\usepackage{amsmath, amsthm, amsfonts, amssymb}
\usepackage[dvipsnames]{xcolor}
\usepackage{tikz-cd}
\usepackage{mathrsfs}
\usepackage{hyperref}
\usepackage{color}
\usepackage[capitalize,nameinlink,noabbrev]{cleveref}
\usepackage[T1]{fontenc}
\newtheorem{theorem}{Theorem}[section]
\newtheorem*{theorem*}{Theorem}
\newtheorem{corollary}[theorem]{Corollary}
\newtheorem*{corollary*}{Corollary}
\newtheorem{proposition}[theorem]{Proposition}
\newtheorem{lemma}[theorem]{Lemma}

\theoremstyle{remark}

\theoremstyle{definition}
\newtheorem{definition}[theorem]{Definition}
\newtheorem{example}[theorem]{Example}

\newcommand{\Hom}{\text{Hom}}
\newcommand{\Ext}{\text{Ext}}

\newcommand{\m}{\mathfrak{m}}
\newcommand{\lbrak}{\left\{\!\!\left\{}
\newcommand{\rbrak}{\right\}\!\!\right\}}
\newcommand{\Z}{\mathbb{Z}}

\newcommand{\f}{\underline{\mathbf{f}}}
\newcommand{\g}{\underline{\mathbf{g}}}
\newcommand{\fR}{\f}
\newcommand{\Ass}{\text{Ass}}

\newcommand{\Supp}{\text{Supp}\,}
\newcommand{\depth}{\text{depth}}

\newcommand{\Spec}{\text{Spec}}
\newcommand{\DDelta}{\Delta\mspace{-12mu}\Delta}
\newcommand{\ncfrob}{\langle F\rangle}
\newcommand{\tT}{\widetilde{T}}
\newcommand{\tS}{\widetilde{S}}

\title{The Fedder action and a simplicial complex of local cohomologies}
\author{Eric Canton and Monica Lewis}
\address{
  Department of Mathematics\\
  University of Michigan\\
  Ann Arbor, MI 48109
}
\subjclass[2010]{Primary: 13D45; Secondary: 13A35, 14B15}

\begin{document}
\maketitle
\begin{abstract}
Let $R$ be a regular ring of prime characteristic $p > 0$, and let $\f=f_1,\ldots,f_c$ be a permutable regular sequence of codimension $c\geq 1$. We describe a complex of $R\langle F \rangle$-modules, denoted $\DDelta^\bullet_{\f}(R)$, whose terms include $\DDelta^0_{\f}(R)=R/\f$ equipped with its natural Frobenius action, and $\DDelta^c_{\f}(R)=H^c_{\f}(R)$ equipped with a Frobenius action we refer to as the Fedder action. We show that $H^i(\DDelta^\bullet_{\f}(R))=0$ for all $i<c$, and that $H^c(\DDelta^\bullet_{\f}(R))$ is a copy of $H^c_{\f}(R)$ equipped with the usual Frobenius action. Using the $\DDelta^\bullet_{\f}(R)$ complex, we show that if $I\supseteq \f$ is an ideal such that $H^i_I(R)=0$ for $\text{ht}(I)<i<\text{ht}(I)+c$ (which is automatic if $R/I$ is Cohen-Macaulay), then the module $H^{\text{ht}(I/\f)+c}_{I/\f}(R/\f)$ has Zariski closed support.
\end{abstract}

\section{Introduction}
Let $R$ be a Noetherian ring of prime characteristic $p > 0$. For any ideal $I\subseteq R$ and any $i\geq 0$, the local cohomology module $H^i_I(R)$ carries a natural action of the Frobenius homomorphism $F:R\to R$. This action is not $R$-linear, but rather, satisfies the relation $F(r\eta)=r^p F(\eta)$ for all $r\in R$ and $\eta\in H^i_I(R)$. This may be viewed as giving $H^i_I(R)$ the structure of a module over the noncommutative ring
\[R\ncfrob:=\frac{R\{F\}}{\left(r^pF-Fr\,|\, r\in R\right)}.\]
From another perspective, to specify a Frobenius action on an $R$-module $M$ is to specify an $R$-linear map $\theta:\mathcal{F}_R(M)\to M$, where $\mathcal{F}_R:\text{Mod}_R\to \text{Mod}_R$ denotes the base change functor along the Frobenius homomorphism. The map $\theta$ is called the \textit{structure morphism} of the action on $M$. 

If $R$ is a regular ring, the natural Frobenius action on $H^i_I(R)$ has a number of desirable properties. First, $H^i_I(R)$ is finitely generated over $R\ncfrob$, e.g., by the image of $\Ext^i_R(R/I,R)$. Second, the structure morphism $\theta:\mathcal{F}_R(H^i_I(R))\to H^i_I(R)$ is an isomorphism. In general, we will refer to an $R\ncfrob$-module as \textit{unit} if its structure map is an isomorphism. Taking the inverse of the structure morphism, a unit $R\ncfrob$-module is precisely the data of an $F_R$-module, in the sense of Lyubeznik \cite{lyufmod}.

To say that a unit $R\ncfrob$-module is finitely generated over $R\ncfrob$ is precisely to say that it is an \textit{$F$-finite} $F_R$-module. Lyubeznik describes a number of strong finiteness properties possessed by a finitely generated unit $R\ncfrob$-module $M$. For example, $M$ has a finite set of associated primes, all unit $R\ncfrob$-submodules of $M$ are finitely generated over $R\ncfrob$ -- in fact, $F$-finite $F_R$-modules are an abelian category -- and for any ideal $I\subseteq R$, $i\geq 0$, the $R\ncfrob$-module structure induced by $M$ makes $H^i_I(M)$ finitely generated. 

The situation is more complicated if the base ring is no longer regular. Consider, for example, the case of a complete intersection ring $S$ cut out in codimension $c$ from a regular ring $R$ of prime characteristic $p>0$. Since $S$ is a homomorphic image of $R$, we may regard a Frobenius action on an $S$-module $M$ as either an $S\ncfrob$-module structure or as an $R\ncfrob$-module structure. We shall prefer the latter. We should not, however, expect the $R\ncfrob$-modules arising as local cohomology modules on $S$, say $H^i_I(S)$, to behave nearly as well as those of the form $H^i_I(R)$. Consider the set of associated primes. When $S$ is a hypersurface -- that is, in codimension $c=1$ -- Katzman showed that $\Ass\,H^i_I(S)$ can be an infinite set \cite{katzinf}. Singh and Swanson demonstrated that this behavior can occur when the hypersurface $S$ is a strongly $F$-regular UFD \cite{singhswan}. Thus, even in the presence of mild singularities, we should therefore expect that either $H^i_I(S)$ is non-finitely generated over $R\ncfrob$, that it is non-unit, or more likely, that both properties fail simultaneously

There are, nonetheless, positive results on the set of minimal primes that are known to hold in this setting. For example, in codimension $c=1$, the support of $H^i_I(S)$ is a Zariski closed set \cite{HochsterNunezBetancourt,KatzmanZhang}. It remains an open question whether the support of the local cohomology of an arbitrary complete intersection ring is always closed \cite{hochreview}.

One may try to use the comparatively well-behaved $R\ncfrob$-module structures on the local cohomology of $R$ as a means to control the local cohomology of $S$. For example, if $J\subseteq R$ is the defining ideal of $S$, then the sequence $0\to J\to R\to S\to 0$ is preserved by the natural Frobenius actions on every term. To control $H^i_I(S)$, Hochster and N\'{u}\~{n}ez-Betancourt consider the following long exact sequence of $R\ncfrob$-modules \cite{HochsterNunezBetancourt}.
\begin{equation}\label{les}
\cdots H^{i}_I(J)\xrightarrow{\alpha} H^{i}_I(R)\to H^i_I(S)\to H^{i+1}_I(J)\to\cdots
\end{equation}
Since $\alpha$ is $R\ncfrob$-linear and $H^i_I(R)$ is finitely generated over $R\ncfrob$, the cokernel of $\alpha$ is finitely generated as well. A finitely generated $R\ncfrob$-module has closed support. If $S$ is presented as a hypersurface, then $J$ is a principal ideal, and so $H^{i+1}_I(J)\cong H^{i+1}_I(R)$ has a finite set of associated primes. Thus, the support of the module $H^i_I(S)$ is closed. Unfortunately, in codimension $c\geq 2$, the set $\Ass\, H^{i+1}_I(J)$ may not necessarily be finite, and its finiteness may in some situations be equivalent to the finiteness of $\Ass\, H^i_I(S)$ \cite{LewisLocalCohomologyParameter}.

Katzman and Zhang give an independent proof of the hypersurface theorem by explicitly describing the supports of the kernel and cokernel of the multiplication by $f$ map on $H^i_I(R)$ \cite{KatzmanZhang}. When $S$ is presented as $R/f$, this is precisely the map $\alpha$ shown in sequence \eqref{les}.

If $\f=f_1,\ldots,f_c$ is a regular sequence of codimension $c$, then for any ideal $I\supseteq \f$ and any $i\geq 0$, there is a natural isomorphism $H^i_I(H^c_{\f}(R))\cong H^{i+c}_I(R)$. The multiplication map $H^i_I(R)\xrightarrow{f} H^i_I(R)$ is the map induced by $H^{i-1}_I(-)$ on $H^1_f(R)\xrightarrow{f} H^1_f(R)$. Instead of starting with the short exact sequence $0\to R\xrightarrow{f} R\to R/f\to 0$, one may therefore consider the long exact sequence of Katzman and Zhang as being induced by
\begin{equation}\label{ses1}
  0\to R/f\to H^1_{f}(R) \xrightarrow{f} H^1_{f}(R)\to 0,
\end{equation}
where $1\in R/f$ is sent to the \v{C}ech cohomology class $\lbrak 1/f\rbrak \in H^1_f(R)$. If one does not wish to compute the supports of the kernel and cokernel of the multiplication-by-$f$ map on $H^i_I(R)$, then we can still obtain the conclusion that $\Supp\, H^i_I(R/f)$ is closed by an argument analogous to Hochster and N\'{u}\~{n}ez-Betancourt so long as the short exact sequence \eqref{ses1} can be made $F$-stable. Indeed, if the right-most copy of $H^1_{f}(R)$ is equipped with the natural Frobenius action $F_{\text{nat}}$ induced by the Frobenius homomorphism of $R$, and the middle copy of $H^1_{f}(R)$ is equipped with the Frobenius action $f^{p-1}F_{\text{nat}}$, then one can verify that all maps in the sequence are $R\ncfrob$-linear.

For a nonzerodivisor $f\in R$, we consider the short exact sequence \eqref{ses1} to be the augmentation of a complex we shall denote by $\DDelta^\bullet_{f}(R)$, whose cohomology is $H^0(\DDelta^\bullet_{f}(R))=0$ and $H^1(\DDelta^\bullet_{f}(R))\simeq H^1_{f}(R)$
\[\DDelta^\bullet_{f}(R):\hspace{1.0em} 0\to R/f\to H^1_{f}(R) \to 0.\]

Our aim in this paper is to obtain a certain higher codimension analogue of the closed support theorems of Hochster and N\'{u}\~{n}ez-Betancourt or Katzman and Zhang by constructing a complex of $R\ncfrob$-modules analogous to $\DDelta^\bullet_f(R)$ in arbitrary codimension.

Let $\f=f_1, \dots, f_c$ be a regular sequence in a regular ring $R$ of prime characteristic $p>0$, corresponding to a complete intersection of codimension $c$. We will sometimes write $\f$ as shorthand for the ideal generated by $\f$, for example, in the notation $H^c_{\f}(R)$ or $R/\f$, though in context, this should not cause any confusion.

Let $f=\prod_{i=1}^c f_i$. As in our discussion on the hypersurface setting, we may embed $R/\f$ into $H^c_{\f}(R)$ by sending $1\in R/\f$ to the \v{C}ech cohomology class $\lbrak 1/f\rbrak$. This embedding is preserved by the Frobenius action $F_{\text{fed}}:=f^{p-1} F_{\text{nat}}$, where we use $F_{\text{nat}}$ to denote the natural action on $H^c_{\fR}(R)$ induced by the Frobenius homomorphism of $R$. We refer to $F_{\text{fed}}$ as the \textit{Fedder action} on $H^c_{\f}(R)$ for reasons that will be described in Section 2. When confusion is possible, we use subscripts $H^c_{\f}(R)_{\text{fed}}$ and $H^c_{\f}(R)_{\text{nat}}$ to distinguish between the $R\ncfrob$-modules obtained by equipping $H^c_{\f}(R)$ with the Fedder action or with the natural action, respectively.

Letting $Q_{\f}$ denote the cokernel of $R/\f\hookrightarrow H^c_{\f}(R)$, we obtain a short exact sequence of $R\ncfrob$-modules,
\[
0\to R/\f \to H^c_{\f}(R)_{\text{fed}}\to Q_{\f}\to 0
\]
For an ideal $I\supseteq \f$, one might hope to control $H^i_I(R/\f)$ via the following long exact sequence, in a manner analogous to Hochster and N\'{u}\~{n}ez-Betancourt.
\begin{equation}\label{les2}
\cdots H^{i-1}_I(H^c_{\f}(R)_{\text{fed}})\xrightarrow{\alpha} H^{i-1}_I(Q_{\f}) \to H^i_I(R/\f)\to H^i_I(H^c_{\f}(R)_{\text{fed}}) \to \cdots
\end{equation}
It is certainly true that $H^i_I(H^c_{\f}(R)_{\text{fed}})$ is isomorphic as an $R$-module to $H^{i+c}_I(R)$, and hence, has a finite set of associated primes. However, significant caution is warranted regarding the cokernel of $\alpha$.

The $R\ncfrob$-module $H^c_{\f}(R)_{\text{fed}}$ is finitely generated -- in fact, it is cyclic, generated by the \v{C}ech cohomology class $\lbrak 1/f^2\rbrak$. However, neither $H^c_{\f}(R)_{\text{fed}}$ nor $Q_{\f}$ are unit $R\ncfrob$-modules. In both cases, the structure morphism is a surjective map with a nontrivial kernel that we compute explicitly in Section 2. So, while Lyubeznik's theorem implies that the iterated local cohomology $H^i_I(H^c_{\f}(R)_{\text{nat}})$ is finitely generated over $R\ncfrob$, we cannot necessarily expect finite generation for $H^i_I(H^c_{\f}(R)_{\text{fed}})$ or $H^i_I(Q_{\f})$. For this reason, we require an alternative to the direct comparison between the local cohomologies of $R/\f$ and $H^c_{\f}(R)_{\text{fed}}$ that is shown in sequence \eqref{les2}. In particular, we would like a complex that allows us to take advantage of both the Fedder action's compatibility with embeddings and the natural action's unit property, as we were able to in the hypersurface case.

Assume that $\f$ is permutable. Let $[c]:=\{1,\ldots,c\}$, and for a subset $T\subseteq [c]$, let $\f_T$ denote the subsequence of $\f$ indexed by the elements of $T$. Thinking in terms of the generating \v{C}ech cohomology classes $\lbrak f_1^{a_1}\cdots f_c^{a_c} \rbrak$ for $a_1,\ldots,a_c\leq -1$, the annihilator in $H^c_{\fR}(R)$ of the subsequence $\f_T$ is generated by those classes with $a_j=-1$ for all $j\in T$, and is isomorphic to a copy of $H^{c-|T|}_{\f_{[c]-T}}(R/\f_T)$. Every such annihilator is preserved by the Fedder action of $H^c_{\f}(R)$.

If we represent the inverse monomials $\lbrak f_1^{a_1}\cdots f_c^{a_c} \rbrak$ for $a_1,\ldots,a_c\leq -1$ as the orthant in the lattice $\Z^c$ consisting of all points $(a_1,\ldots,a_c)$ with $a_1,\ldots,a_c\leq -1$, then the annihilators of subsequences $\f_T$ correspond to exterior faces of the orthant. This visibly simplicial structure suggests the arrangement of the local cohomologies $H^{c-|T|}_{\f_{[c]-T}}(R/\f_T)$ for $T\subseteq [c]$ into a complex $\DDelta_{\fR}^\bullet(R)$ of length $c$. The degree $i$ term $\DDelta^i_{\fR}(R)$ is the direct sum of the annihilators of ideals $\f_T$ for $|T|=c-i$, and the differentials are given by direct sums of signed inclusion maps (similar to Koszul or \v{C}ech complexes). At one end of the complex, we have $\DDelta^0_{\fR}(R)=R/{\fR}$, and at the other, we have $\DDelta^c_{\fR}(R)=H^c_{\fR}(R)$. Each term may be equipped with a Fedder action in order to make all differentials in the complex $R\ncfrob$-linear.

Consider, for example, the codimension $2$ setting. Given a regular sequence $f,g\in R$, the complex $\DDelta^\bullet_{f,g}(R)$, shown below with its augmentation map, has the form
\[0\to R/(f,g)R\to H^1_{f}(R/g)\oplus H^1_{g}(R/f)\to H^2_{f,g}(R)_{\text{fed}}\xrightarrow{fg} H^2_{f,g}(R)_{\text{nat}}\to 0\]

The Fedder action on $H^1_f(R/g)$ is $f^{p-1}F_{\text{nat}}$ and the Fedder action on $H^1_{g}(R/f)$ is $g^{p-1}F_{\text{nat}}$. With the Fedder and natural actions on the left and right copy of $H^c_{\f}(R)$, respectively, all maps in this complex are $R\ncfrob$-linear. The cohomology of $\DDelta^\bullet_{f,g}(R)$ is straightforward to compute directly.

In Section 4, we compute the cohomology of the $\DDelta^\bullet_{\fR}(R)$ complex in general, showing that $H^i(\DDelta_{\fR}^\bullet(R))=0$ for $0\leq i<c$, and that $H^c(\DDelta_{\fR}^\bullet(R))\simeq H^c_{\fR}(R)$. The induced Frobenius action on the augmentation $H^c(\DDelta_{\fR}^\bullet(R))$ is the natural action induced by the Frobenius homomorphism of $R$ -- it is not the Fedder action. As an application of the exactness of the augmented $\DDelta^\bullet_{\fR}(R)$ complex, we show in Section 5 that if $I$ is an ideal containing ${\fR}$ that satisfies the vanishing condition $H^i_I(R)=0$ for $\text{ht}(I) < i < \text{ht}(I)+c$, then the module $H^{\text{ht}(I/{\fR})+c}_{I/{\fR}}(R/{\fR})$ has Zariski closed support.

\section{The Fedder Action}
Suppose that $R$ is a Noetherian ring of characteristic $p > 0$ and $J \subset R$ is an ideal. Recall the notation $J^{[p]}$ for the expansion of $J$ under the Frobenius map $F: R \to R$, $F(r) = r^p$. Explicitly, $J^{[p]}$ is the ideal of $R$ generated by $\{f^p \,:\, f \in J\}$. Similar, $J^{[p^e]}$ is the ideal generated by $\{f^{p^e} \,:\, f \in J\}$ for any $e \ge 1$. We denote by $F^e_*R$ the codomain of $F^e: R \to R$, viewed as a $R$-algebra. An element $r\in R$ acts on $s \in F^e_*R$ via $r\cdot s=r^{p^e}s$.

The ring $R$ is \textit{$F$-finite} if $F:R\to R$ is a finite map. This property is preserved by localization, homomorphic images, and finite type extension.

\subsection{The Fedder action associated with a Gorenstein local ring}

To motivate the definition of the Fedder action, we consider the setting of an $F$-finite Gorenstein local ring. The main property we require is as follows.

\begin{lemma}\label{principal Cartier algebra}
  Let $(S, \m)$ be an $F$-finite Gorenstein local ring of characteristic $p > 0$. For some non-zero $T: F_*^e S \to S$, we have an isomorphism of $F^e_*S$-modules. 
  \[ \Hom_S(F_*^eS, S) \cong T \cdot (F_*^e S) \]
\end{lemma}
\begin{proof}
  Since $F^e_*S$ is finitely generated as an $S$-module, and isomorphic to $S$ as a ring, we have a finite, local extension of Gorenstein rings $(S, \m) \subset (F_*^eS, F_*^e\m)$. Thus, $\Hom_S(F_*^eS, S) \cong \Hom_S(F_*^eS, \omega_S) \cong \omega_{F_*^eS} \cong F_*^eS$ as $F^e_*S$-modules. 
\end{proof}

When $S=R/J$ is Gorenstein for $R$ an $F$-finite regular local ring, the fact that $\Hom_S(F_*S,S)$ is principal leads to the following well-known description of $(J^{[p]}:J)$, which plays a key role in the proof of Fedder's criterion \cite{Fedder83}.

\begin{lemma}
  \label{principal Fedder socle}
  Suppose $R$ is regular and $J \subset R$ is an ideal with $R/J$ Gorenstein. Then for some $g \in R$, $(J^{[p]} :J) = gR + J^{[p]}$, and $(J^{[p^e]}:J^{[p^{e-1}]}) = g^{p^{e-1}}R + J^{[p^e]}$ for all $e \ge 1$. 
\end{lemma}
\begin{proof}
The assumption that $R$ is regular implies that $F_*R$ is flat over $R$, hence $(J^{[p^e]}:_R J^{[p^{e-1}]}) = (J^{[p]} : J)^{[p^{e-1}]}$ for all $e \ge 1$. Thus, it suffices to prove $(J^{[p]}: J) = gR + J^{[p]}$. Indeed, if we fix a generator $T$ for $\Hom_R(F_*R, R)$ over $F_*R$, the map
  \[ (J^{[p]}:J)/J^{[p]} \xrightarrow{\sim} \Hom_{R/J}(F_*(R/J), (R/J)) \]
  sending $(r + J^{[p]})$ to $(T \cdot r) + J$ is an isomorphism \cite{Fedder83}. By \cref{principal Cartier algebra}, we know $\Hom_{R/J}(F_*(R/J), R/J)$ is principal over $F_*(R/J)$, say via $(T \cdot g) + J$ for $g \in (J^{[p]} : J)$. We conclude $(J^{[p]} : J) = gR + J^{[p]}$. 
\end{proof}

Let $R/J$ be a Gorenstein quotient with $R$ regular, and fix a generator $g$ for $(J^{[p]}: J)$. Lemma \ref{principal Fedder socle} allows us to define a directed system

\begin{equation}\label{dirsys}
\begin{tikzcd}
  0 \ar[r] & R/J \ar[r, "g"] & R/J^{[p]} \ar[r, "g^p"] & R/J^{[p^2]} \ar[r, "g^{p^2}"] & R/J^{[p^3]} \ar[r] & \cdots 
\end{tikzcd}
\end{equation}
whose transition maps are injective.

\begin{definition}
Let $R$ be a ring of prime characteristic $p>0$ and let $M$ be an $R$-module. A \textit{Frobenius action} on $M$ is an additive map $\beta:M\to M$ such that $\beta(r\eta) = r^p\beta(\eta)$ for all $r\in R$ and all $\eta \in M$. The Frobenius homomorphism $F:R\to R$ is the \textit{natural action} on $R$. If $M$ and $N$ are $R$-modules equipped with Frobenius actions $\alpha$ and $\beta$, respectively, then an $R$-linear map $h:M\to N$ is called \textit{Frobenius stable} if $h\circ \beta = \alpha\circ h$. If $h:M\to N$ is Frobenius stable, then both $\text{Ker}(h)$ and $\text{Coker}(h)$ inherit an \textit{induced action} from the actions on $M$ and $N$, respectively.
\end{definition}

Let $M = \varinjlim_e (R/J^{[p^e]}, g^{p^e})$ denote the direct limit of the system. The embedding $R/J\hookrightarrow M$ can be made Frobenius stable with respect to the natural action of $R/J$. Specifically, let $M$ be equipped with the action $\varphi:M\to M$ described by $gF:R/J^{[p^e]}\to R/J^{[p^{e+1}]}$ at the level of its defining directed system. The compatibility with the natural action $F:R/J\to R/J$ is shown below.
\[
\begin{tikzcd}
  & 0 \ar[r] & R/J\ar[dl,"F"'] \ar[r, "g"]\ar[d, "gF"] & R/J^{[p]} \ar[r, "g^p"]\ar[d, "gF"] & R/J^{[p^2]} \ar[r, "g^{p^2}"]\ar[d, "gF"] & R/J^{[p^3]} \ar[r]\ar[d, "gF"] & \cdots\ar[r] & M\ar[d,"\varphi"]\\
  0 \ar[r] & R/J \ar[r, "g"] & R/J^{[p]} \ar[r, "g^p"] & R/J^{[p^2]} \ar[r, "g^{p^2}"] & R/J^{[p^3]} \ar[r,"g^{p^3}"] & R/J^{[p^4]} \ar[r] & \cdots\ar[r] & M
\end{tikzcd}
\]

We refer to the resulting action $\varphi:M\to M$ on $M$ as the {\em Fedder action}.

\subsection{The Fedder action associated with a regular sequence}
For a regular sequence $\f=f_1,\ldots,f_c$, by an abuse of notation, we will use $\f^{[t]}$ to denote the sequence $f_1^t,\ldots,f_c^t$ regardless of whether $t$ is a power of the characteristic. Let $f=\prod_{i=1}^c f_i$.

When the ideal $J$ in the directed system \eqref{dirsys} is generated by a regular sequence $\f=f_1,\ldots,f_c$, the Fedder socle of $R/{\f^{[p]}}$ has a clear choice of generator: $f^{p-1}$. Moreover, the direct limit \eqref{dirsys} is identifiable as $H^c_{\fR}(R)$. For $q=p^e$, the action sending $r+\f^{[q]}\in R/\f^{[q]}$ to $f^{p-1}r^p+\f^{[qp]}\in R/\f^{[qp]}$ at the level of the directed system $(R/\f^{[q]},f^{qp-q})_{e=0}^\infty$ has the form
\[\lbrak\frac{r}{f^{q}} \rbrak \mapsto \lbrak\frac{r^pf^{p-1}}{f^{qp}} \rbrak \]
on \v{C}ech classes in $H^c_{\f}(R)$.

While our primary motivation is the case in which $R$ is a regular ring, the colon properties 
\[(\f^{[b]}:f^{b-a})=\f^{[a]}\hspace{1.0em}\text{and}\hspace{1.0em}(\f^{[b]}:\f^{[a]})=f^{b-a}R+\f^{[b]}\]
for two positive integers $b>a$ hold for an arbitrary regular sequence $\f$ in a Noetherian ring. Many properties of the constructions that follow therefore work in greater generality than the motivating setting of the previous section. 
\begin{definition}
Let $R$ be a Noetherian ring of prime characteristic $p>0$, let $\f=f_1,\ldots,f_c\in R$ be a regular sequence of codimension $c$, and let $f=\prod_{i=1}^c f_i$. The \textit{natural action} on $H^c_{\f}(R)$, denoted $F_{\text{nat}}$, is the map $H^c_{\f}(F):H^c_{\f}(R)\to F_*H^c_{\f}(R)$ induced by $H^c_{\f}(-)$ on the Frobenius homomorphism $F:R\to F_*R$, via the natural identification of $R$-modules $H^c_{\f}(F_*R)=F_* H^c_{\f^{[p]}}(R)=F_* H^c_{\f}(R)$. The \textit{Fedder action with respect to $\f$} on $H^c_{\f}(R)$ -- or simply \textit{the Fedder action}, if the sequence $\f$ is understood -- is defined by $F_{\text{fed}}:=f^{p-1} F_{\text{nat}}$. When there is risk of confusion, we write $H^c_{\f}(R)_{\text{nat}}$ and $H^c_{\f}(R)_{\text{fed}}$ to denote $H^c_{\f}(R)$ equipped with the natural action and the Fedder action, respectively.
\end{definition}

Representing elements of $H^c_{\f}(R)$ as \v{C}ech cohomology classes with respect to the sequence $\f$, we have for $r\in R$ and $q=p^e$,
\[F_{\text{nat}}:\lbrak\frac{r}{f^{q}} \rbrak \mapsto \lbrak\frac{r^p}{f^{{qp}}}\rbrak\]
If $R\ncfrob$ denotes the noncommutative ring $R\{F\}/(r^pF-Fr\,|\, r\in R)$, then the class $\lbrak 1/f\rbrak$ generates $H^c_{\f}(R)_{\text{nat}}$ as an $R\ncfrob$-module. It is still true that $H^c_{\f}(R)_{\text{fed}}$ is a cyclic $R\ncfrob$-module, but we require a different generator. This is because,
\[ F_{\text{fed}}:\lbrak\frac{r}{f} \rbrak \mapsto \lbrak\frac{r^p}{f} \rbrak \]
for any $r\in R$. That is, the image of the map $R/\f\to H^c_{\f}(R)$ sending $1\mapsto \lbrak 1/f\rbrak$ is preserved by the Fedder action. Notice that for each $q=p^e$ we have
\[F_{\text{fed}}:\lbrak\frac{r}{f^{q+1}} \rbrak \mapsto \lbrak\frac{r^p}{f^{{qp+1}}} \rbrak \]
and thus, the class $\lbrak1/f^2\rbrak$ generates $H^c_{\f}(R)_{\text{fed}}$ over $R\ncfrob$.

Perhaps the most useful property of the Fedder action is its compatibility with embeddings of annihilators of subsequences, in the sense of the following proposition. This compatibility was observed in \cite{CantonGradedCIs}, essentially as the result of applying $H^*_{\f}$ to the Koszul complex $K^\bullet(\g;R)$.
\begin{proposition}\label{fedemb}
  Let $R$ be a Noetherian ring of prime characteristic $p>0$, let $g_1,\ldots,g_t,f_1,\ldots,f_c\in R$ be a regular sequence, and write $\g=g_1,\ldots,g_t$, $\f=f_1,\ldots,f_c,$, $g=\prod_{i=1}^t g_i$ and $f=\prod_{i=1}^c f_i$. Consider $H^c_{\f}(R/\g)$ and $H^{t+c}_{\g,\f}(R)$ as $R\ncfrob$-modules via the Fedder actions with respect to $\f$ and $\g,\f$, respectively. There is an $R\ncfrob$-linear injection
  \[
  H^c_{\f}(R/\g)\hookrightarrow H^{t+c}_{\g,\f}(R)
  \]
whose image is the annihilator $(0:_{H^{t+c}_{\g,\f}(R)} \g)$.
\end{proposition}
\begin{proof}
Since $(\g^{[a]},\f^{[a]}):g^{a-1}= (\g,\f^{[a]})$ for all $a\geq 1$, multiplication by $g^{a-1}$ induces a well-defined injection $\phi_a:R/(\g,\f^{[a]})\xrightarrow{g^{a-1}}R/(\g^{[a]},\f^{[a]})$, and since $(\g^{[a]},\f^{[a]}):\g=g^{a-1}R+(\g^{[a]},\f^{[a]})$, the image of $\phi_a$ is precisely $(0:_{R/(\g^{[a]},\f^{[a]})}\g)$. The maps $\phi_a$ form a map of directed systems $(R/(\g,\f^{[a]}),f)_{a=1}^\infty$ to $(R/(\g^{[a]},\f^{[a]}),gf)_{a=1}^\infty$ via
\[
\begin{tikzcd}
  R/(\g,\f^{[a]})\ar[r,"g^{a-1}"]\ar[d,"f^{b-a}"] & R/(\g^{[a]},\f^{[a]})\ar[d,"(gf)^{b-a}"]\\
  R/(\g,\f^{[b]})\ar[r,"g^{b-1}"] & R/(\g^{[b]},\f^{[b]})
\end{tikzcd}
\]
On the direct limits, this produces an injection $\phi:H^{c}_{\f}(R/\g)\hookrightarrow H^{t+c}_{\g,\f}(R)$, whose image is precisely $(0:_{H^{t+c}_{\g,\f}(R)}\g)$. For each $q=p^e$, the Fedder action with respect to $\f$ sends the class $r+(\g,\f^{[q]})\in R/(\g,\f^{[q]})$ to $f^{p-1}r^p+(\g,\f^{[qp]})\in R/(\g,\f^{[qp]})$, which is sent by $\phi_{qp}$ to $g^{qp-1}(f^{p-1}r^p+(\g^{[qp]},\f^{[qp]})\in R/(\g^{[qp]},\f^{[qp]})$. On the other hand, $\phi_q$ sends $r+(\g,\f^{[q]})$ to $g^{q-1}r+(\g^{[q]},\f^{[q]})$, and the Fedder action with respect to $\g,\f$ sends $g^{q-1}r+(\g^{[q]},\f^{[q]})$ to
\[
(gf)^{p-1}\left(g^{q-1}r\right)^p+(\g^{[qp]},\f^{[qp]}) = g^{qp-1}(f^{p-1}r^p)+(\g^{[qp]},\f^{[qp]}),
\]
as desired.
\end{proof}

\subsection{The structure morphism of the Fedder action}
A Frobenius action is equivalent data to specifying a structure morphism in the following sense.

\begin{definition}
Let $R$ be a ring of prime characteristic $p>0$ and let $M$ be an $R$-module equipped with a Frobenius action $\beta:M\to M$. If the codomain of $\beta$ is regarded as an $R$-module via restriction of scalars along the Frobenius homomorphism of $R$, then the resulting map $\beta:M\to F_* M$ is $R$ linear. This map induces an $F_*R$-linear map $\theta: F_*R\otimes_R M\to F_*M$ sending $s\otimes \eta \mapsto s\beta(\eta)$ for $s\in F_*R$ and $\eta \in M$. If $\mathcal{F}_R(-)$ denotes the base change functor along the Frobenius homomorphism of $R$, then $\theta$ may be regarded as an $R$-linear map $\theta:\mathcal{F}_R(M)\to M$. We call this map the \textit{structure morphism} of the Frobenius action $\beta$. Call $\beta$ a \textit{unit} action (using the terminology of \cite{blickle2005local}) if the corresponding structure morphism $\theta:\mathcal{F}_R(M)\to M$ is an isomorphism.
\end{definition}

Given an arbitrary $R$-linear map $\theta:\mathcal{F}_R(M)\to M$, which we may consider to be an $F_*R$-linear map $F_*R\otimes_R M\to F_*M$, we can recover the corresponding Frobenius action on $M$ by first mapping $M\to F_*R\otimes_R M$ via $\eta\mapsto 1\otimes \eta$, and letting $F(\eta)$ be the element $\theta(1\otimes \eta)$ upon identifying $F_*M$ with $M$ as abelian groups. If we think of $F_*R$ as $R^{1/p}$, with $F_*M=M^{1/p}$, then the $Z$-linear identification $M^{1/p}\to M$ may be regarded as a formal $p$th power map.

If $\mathcal{F}_R(R)$ is identified with $R$ in the natural way, then the structure morphism $\mathcal{F}_R(R)\to R$ of the natural action on $R$ is the identity map $R\xrightarrow{1} R$. For an ideal $I\subseteq R$, the Frobenius homomorphism on $R/I$ can be understood as an action either over $R$ or over $R/I$. In the former case, the structure morphism $\mathcal{F}_R(R/I)=R/I^{[p]}\to R/I$ is the quotient map by $I/I^{[p]}$. In the latter case, the structure morphism $\mathcal{F}_{R/I}(R/I)=R/I\to R/I$ is the identity. To avoid ambiguity, we will refer to the former as the $R\ncfrob$ structure morphism and the latter as the $(R/I)\ncfrob$ structure morphism.

If the Frobenius homomorphism of $R$ is flat -- equivalently, if $R$ is regular \cite{Kunz1969} -- then $\mathcal{F}_R(-)$ is exact, and for any ideal $I\subseteq R$ and any $i\geq 0$, the module $\mathcal{F}_R(H^i_I(R))$ is canonically identified with $H^i_I(R)$. Under this identification, the structure morphism $\mathcal{F}_R(H^i_I(R))\to H^i_I(R)$ of the natural action is the identity map $H^i_I(R)\xrightarrow{1} H^i_I(R)$. The natural action on $H^i_I(R)$ is therefore a unit action.

Let $R$ be regular, let $\f=f_1,\ldots,f_c\in R$ be a regular sequence of codimension $c$, and let $Q_{\f}$ denote the cokernel of the Frobenius stable embedding $R/\f\hookrightarrow H^c_{\f}(R)_{\text{fed}}$ that sends $1\mapsto \lbrak 1/f\rbrak$, equipped with its induced action. There is a Frobenius stable short exact sequence
\begin{equation}\label{comp1}
  0\to R/\f\to H^c_{\f}(R)_{\text{fed}}\to Q_{\f}\to 0
\end{equation}
A \textit{Frobenius action on a complex} $A^\bullet$ is a choice of Frobenius action on each term $A^i$ such that the differentials $d^i:A^i\to A^{i+1}$ are Frobenius stable. Analogous to the situation with modules, the data of a Frobenius action on $A^\bullet$ is equivalent to specifying $R$-linear map of complexes $\Theta:\mathcal{F}_R(A^\bullet)\to A^\bullet$ -- the \textit{structure morphism} of the complex. In this section, we will describe the structure morphism of the three-term complex \eqref{comp1}.

\begin{theorem}\label{structuremorphism}
  Let $R$ be a regular ring of prime characteristic $p>0$, let $\f=f_1,\ldots,f_c\in R$ be a regular sequence, let $f=\prod_{i=1}^c f_i$, and let $A^\bullet$ denote the complex described in \eqref{comp1} corresponding to the Frobenius stable embedding $R/\f\to H^c_{\f}(R)_{\text{fed}}$. Let $\theta_{R/\f}$, $\theta_{\text{fed}}$, and $\theta_Q$ denote the $R\ncfrob$ structure morphisms of $R/\f$, $H^c_{\f}(R)_{\text{fed}}$, and $Q_{\f}$, respectively. Let $\Theta$ denote the structure morphism of $A^\bullet$. There is an exact sequence of complexes
  \[
  0\to K^\bullet \to \mathcal{F}_R(A^\bullet)\xrightarrow{\Theta} A^\bullet\to 0
  \]
  described term-by-term as follows, where $\mathcal{F}_R(H^c_{\f}(R))$ is identified with $H^c_{\f}(R)$ in the natural way:
\[
\begin{tikzcd}
0 & & & 0 & 0 & 0 &\\
A^\bullet\ar[u]&\hspace{-3.0em}:&  0\ar[r]& R/\f\ar[r]\ar[u]& H^c_{\f}(R)\ar[r]\ar[u]& Q_{\f}\ar[r]\ar[u]& 0\\
\mathcal{F}_R(A^\bullet)\ar[u,"\Theta"]&\hspace{-3.0em}:&  0\ar[r]& R/\f^{[p]}\ar[u,"\theta_{R/\f}"]\ar[r]& H^c_{\f}(R)\ar[r]\ar[u,"\theta_{\text{fed}}"]& \mathcal{F}_R(Q_{\f})\ar[r]\ar[u,"\theta_Q"]& 0\\
K^\bullet\ar[u]&\hspace{-3.0em}:& 0\ar[r] & \f/\f^{[p]}\ar[r]\ar[u] & (0:_{H^c_{\f}(R)} f^{p-1})\ar[r]\ar[u] & V_{\f} \ar[r]\ar[u] & 0\\
0\ar[u]&  & & 0\ar[u] & 0\ar[u] & 0\ar[u] &
\end{tikzcd}
\]
The module $V_{\f}:=\text{Ker}(\theta_{Q})$ can be described by the direct limit
\[
\cdots\to\frac{\f^{q+1}}{f^qR+\f^{[q+p]}}\xrightarrow{f^{qp-q}}\frac{\f^{qp+1}}{f^{qp}R+\f^{[qp+p]}}\to\cdots \to V_{\f}
\]
\end{theorem}
\begin{proof}
We first describe $A^\bullet$ as a direct limit of complexes $\varinjlim_e(A^\bullet_e,\psi_e)$,
\[
\begin{tikzcd}
A^\bullet&\hspace{-3.0em}:&  0\ar[r]& R/\f\ar[r]& H^c_{\f}(R)\ar[r]& Q_{\f}\ar[r]& 0\\
\vdots\ar[u]&& &\vdots\ar[u] & \vdots\ar[u] & \vdots\ar[u] &\\
A^\bullet_{e+1}\ar[u]&\hspace{-3.0em}: &  0\ar[r]& R/\f\ar[u]\ar[r,"f^{qp}"]& R/\f^{[qp+1]}\ar[r]\ar[u]& R/(f^{qp}R + \f^{[qp+1]})\ar[r]\ar[u]& 0\\
A^\bullet_e\ar[u,"\psi_e"]&\hspace{-3.0em}:&  0\ar[r]& R/\f\ar[u,"1"]\ar[r,"f^{q}"]& R/\f^{[q+1]}\ar[u,"f^{qp-q}"]\ar[r]& R/(f^{q}R + \f^{[q+1]})\ar[u,"f^{qp-q}"]\ar[r]& 0\\
\vdots\ar[u]&&&\vdots\ar[u] & \vdots\ar[u] & \vdots\ar[u] & 
\end{tikzcd}
\]
In the direct limit, an element $r+\f^{[q+1]}\in R/\f^{[q+1]}$ maps to the \v{C}ech cohomology class $\lbrak r/f^{q+1}\rbrak$. To describe the structure morphism $\Theta:\mathcal{F}_R(A^\bullet)\to A^\bullet$, we specify a map $\Theta_{e-1}:\mathcal{F}_R(A^\bullet_{e-1})\to A^\bullet_{e}$ for each $e$ by taking the obvious quotient maps term-by-term.
\[
\begin{tikzcd}
A^\bullet_{e}&\hspace{-3.0em}: & 0\ar[r]& R/\f\ar[r,"f^{q}"]& R/\f^{[q+1]}\ar[r]& R/(f^{q}R + \f^{[q+1]})\ar[r]& 0\\
\mathcal{F}_R(A^\bullet_{e-1})\ar[u,"\Theta_{e-1}"]&\hspace{-3.0em}: & 0\ar[r]& R/\f^{[p]}\ar[u,two heads]\ar[r,"f^{q}"]& R/\f^{[q+p]}\ar[u, two heads]\ar[r]& R/(f^{q}R + \f^{[q+p]})\ar[u, two heads]\ar[r]& 0
\end{tikzcd}
\]
It is straightforward to verify that the following diagram of complexes commutes
\[
\begin{tikzcd}
  \mathcal{F}_R(A^\bullet_e)\ar[r,"\Theta_e"] & A^\bullet_{e+1}\\
  \mathcal{F}_R(A^\bullet_{e-1})\ar[u,"\mathcal{F}_R(\psi_{e-1})"]\ar[r,"\Theta_{e-1}"] & A^\bullet_{e}\ar[u,"\psi_e"]
\end{tikzcd}
\]
so that the $(\Theta_e)_{e=1}^\infty$ induce a well-defined map on the direct limit $\mathcal{F}_R(A^\bullet)\to A^\bullet$. To describe the Frobenius action on the class $r+\f^{[q+1]}$, the inclusion $R/\f^{[q+1]}\to F_*R\otimes_R (R/\f^{[q+1]})$ sends $r+\f^{[q+1]}$ to $r^p + \f^{[qp+p]}$ once the codomain is identified with $R/{\f^{[qp+p]}}$. The quotient $R/{\f^{[qp+p]}}\twoheadrightarrow R/{\f^{[qp+1]}}$ sends that class to $r^p+\f^{[qp+1]}$. On the corresponding \v{C}ech cohomology classes in the direct limit, we have
\[
\lbrak\frac{r}{f^{q+1}}\rbrak\mapsto \lbrak\frac{r^p}{f^{qp+1}}\rbrak = \lbrak\frac{f^{p-1}r^p}{f^{qp+p}}\rbrak = F_{\text{fed}}\lbrak\frac{r}{f^{q+1}}\rbrak
\]
as desired. Concerning the kernel of the structure map, we have for each $e$ a commutative diagram
\[
\begin{tikzcd}
  0 & & & 0 & 0 & 0 &\\
A^\bullet_{e}\ar[u]&\hspace{-3.0em}: & 0\ar[r]& R/\f\ar[r,"f^{q}"]\ar[u]& R/\f^{[q+1]}\ar[r]\ar[u]& R/(f^{q}R + \f^{[q+1]})\ar[r]\ar[u]& 0\\
\mathcal{F}_R(A^\bullet_{e-1})\ar[u,"\Theta_{e-1}"]&\hspace{-3.0em}: & 0\ar[r]& R/\f^{[p]}\ar[u]\ar[r,"f^{q}"]& R/\f^{[q+p]}\ar[u]\ar[r]& R/(f^{q}R + \f^{[q+p]})\ar[u]\ar[r]& 0\\
K^\bullet_{e}\ar[u]&\hspace{-3.0em}: & 0\ar[r]& \f/\f^{[p]}\ar[u]\ar[r,"f^{q}"]& \f^{[q+1]}/\f^{[q+p]}\ar[u]\ar[r]& (f^{q}R + \f^{[q+1]})/(f^{q}R + \f^{[q+p]})\ar[u]\ar[r]& 0\\
0\ar[u] & & & 0\ar[u] & 0\ar[u] & 0\ar[u] &
\end{tikzcd}
\]
where $\mathcal{F}_R(\psi_{e-1})$ induces a map $K^\bullet_e\to K^\bullet_{e+1}$
\[
\begin{tikzcd}
K^\bullet_{e+1}&\hspace{-3.0em}: & 0\ar[r]& \f/\f^{[p]}\ar[r,"f^{qp}"]& \f^{[qp+1]}/\f^{[qp+p]}\ar[r]& \f^{[qp+1]}/(f^{qp}R + \f^{[qp+p]})\ar[r]& 0\\
K^\bullet_{e}\ar[u]&\hspace{-3.0em}: & 0\ar[r]& \f/\f^{[p]}\ar[u,"1"]\ar[r,"f^{q}"]& \f^{[q+1]}/\f^{[q+p]}\ar[u,"f^{qp-q}"]\ar[r]& \f^{[q+1]}/(f^{q}R + \f^{[q+p]})\ar[u,"f^{qp-q}"]\ar[r]& 0
\end{tikzcd}
\]
compatible with the rest of the directed system. Note that $\f^{[q+1]}/\f^{[q+p]} = (\f^{[q+p]}:f^{p-1})/\f^{[q+p]}$ for each $q=p^e$, so that in the direct limit, $\text{Ker}(\theta_{\text{fed}}) = (0:_{H^c_{\f}(R)} f^{p-1})$.
\end{proof}

Note that the kernel $(0:_{H^c_{\f}(R)} f^{p-1})$ of $\theta_{\text{fed}}$ can be explicitly described as a local cohomology module, namely, it is isomorphic to $H^{c-1}_{\f}(R/f^{p-1})$.
\begin{proposition}
  Let $R$ be a Noetherian ring, let $\f=f_1,\ldots,f_c\in R$ be a regular sequence of codimension $c$, and let $h\in R$ be a nonzerodivisor. Then
  \[
  (0:_{H^c_{\f}(R)} h)\cong H^{c-1}_{\f}(R/h).
  \]
\end{proposition}
\begin{proof}
Consider the double complex $(0\to R\xrightarrow{h} R\to 0)\otimes_R \check{C}^\bullet(\f;R)$.
  \[
\begin{tikzcd}
    & 0\ar[d] &  & 0\ar[d] & 0\ar[d] &\\
  0\ar[r] &  \check{C}^1(\f;R)\ar[d,"h"]\ar[r] & \cdots\ar[r] & \check{C}^{c-1}(\f;R)\ar[d,"h"]\ar[r] & \check{C}^c(\f;R)\ar[d,"h"]\ar[r] & 0\\
  0\ar[r] &  \check{C}^1(\f;R)\ar[r]\ar[d] & \cdots\ar[r] & \check{C}^{c-1}(\f;R)\ar[r]\ar[d] & \check{C}^c(\f;R)\ar[r]\ar[d] & 0\\
  & 0 &  & 0 & 0 &
\end{tikzcd}
\]
If we compute cohomology first horizontally and then vertically, we see that the cohomology of the totalization in degree $c$ is $\text{Hom}_R(R/h,H^c_{\f}(R))$, since $E_2=E_\infty$ and the first $c-1$ columns vanish. On the other hand, if we first take cohomology vertically and then horizontally, then $E_2=E_\infty$ and the first row vanishes, so the cohomology of the totalization in degree $c$ is also isomorphic to $H^{c-1}(R/h\otimes_R \check{C}^\bullet(\f;R))=H^{c-1}_{\f}(R/h)$, as desired.
\end{proof}

The kernel of $\theta_Q$ has a particularly nice description in codimension $2$, where it decomposes as a direct sum of two local cohomology modules, $V_{f,g}\cong H^1_f(R/g^{p-1})\oplus H^1_g(R/f^{p-1})$.

\begin{corollary}
  Let $R$ be a Noetherian ring of prime characteristic $p>0$, let $f,g\in R$ be a regular sequence, and let $V_{f,g}$ denote the following direct limit over all $q=p^e$
\[
\cdots\to\frac{(f^{q+1},g^{q+1})}{((fg)^q,f^{q+p},g^{q+p})}\xrightarrow{f^{qp-q}}\frac{(f^{qp+1},g^{qp+1})}{((fg)^{qp},f^{qp+p},g^{qp+p})}\to\cdots \to V_{f,g}
\]
Then
    \[
V_{f,g}\cong H^1_g(R/f^{p-1})\oplus H^1_f(R/g^{p-1})
  \]
\end{corollary}
\begin{proof}
Note that if $af^{q+1}=bg^{q+1}$ mod $((fg)^q,f^{q+p},g^{q+p})$ for some $a,b\in R$, then
\begin{align*}
  a\in((fg)^q,f^{q+p},g^{q+1}):f^{q+1} &= \left((f^{q+p},g^{q+1}):(f^p,g)\right) : f^{q+1}\\
  &=\left((f^{q+p},g^{q+1}): f^{q+1}\right) :(f^p,g)\\
  &=\left(f^{p-1},g^{q+1}\right) :(f^p,g)\\
  &=\left(f^{p-1},g^q\right)
\end{align*}
so that $af^{q+1} \in (f^{q+p},f^{q+1}g^q)$, which is zero mod $((fg)^q,f^{q+p},g^{q+p})$. Thus, the generators $u_{1,e}:=f^{q+1}$ and $u_{2,e}:=g^{q+1}$ of $(f^{q+1},g^{q+1})/((fg)^q,f^{q+p},g^{q+p})$ have $R$-spans with an intersection of $0$, yielding a direct sum
  \[
  \frac{(f^{q+1},g^{q+1})}{((fg)^q,f^{q+p},g^{q+p})}\cong \frac{R u_{1,e}}{g^q u_{1,e}R + f^{p-1} u_{1,e} R}\oplus \frac{R u_{2,e}}{f^q u_{2,e}R + g^{p-1} u_{2,e} R}
  \]
  The transition map $(fg)^{qp-q}$ sends $f^{q+1}=u_{1,e}$ to $g^{qp-q}f^{qp+1}=g^{qp-q}u_{1,e+1}$, and likewise, $u_{2,e}\mapsto f^{qp-q} u_{2,e+1}$, breaking into the direct sum of transition maps on the $u_1$ and $u_2$ components,
  \[
  \cdots\to\frac{R u_{1,e}}{g^q u_{1,e}R + f^{p-1} u_{1,e} R}\xrightarrow{g^{qp-q}}
  \frac{R u_{1,e+1}}{g^{qp} u_{1,e+1}R + f^{p-1} u_{1,e+1} R}\to\cdots\to H^1_g(R/f^{p-1})
  \]
  and
  \[
  \cdots\to\frac{R u_{2,e}}{f^q u_{2,e}R + g^{p-1} u_{2,e} R}\xrightarrow{f^{qp-q}}
  \frac{R u_{2,e+1}}{f^{qp} u_{2,e+1}R + g^{p-1} u_{2,e+1} R}\to\cdots\to H^1_f(R/g^{p-1})
  \]
as desired.
\end{proof}

\section{A Simplicial Complex of Local Cohomology Modules}\label{simplicial section}

Assume from this point onward that all regular sequences $\f$ are permutable, by which we mean that $f_{\sigma(1)},\ldots,f_{\sigma(c)}$ is a regular sequence for all $\sigma\in S_c$. This is automatic in the local case, or in the standard graded case assuming that $\f$ is homogeneous. Let $[c]:=\{1,\ldots,c\}$, and for a subset $T\subseteq [c]$, let $\f_T$ denote the subsequence of $\f$ indexed by the elements of $T$. The permutability of $\f$ is equivalent to assuming that $\f_T$ is a regular sequence for all subsets $T\subseteq [c]$. Let $\tT=[c]\setminus T$, with $\f_{\tT}$ the complementary subsequence to $\f_T$ in $\f$. 

\subsection{The Chain Complex and Filtration.}

Let $R$ be a Noetherian ring, and fix a permutable regular sequence $\f=f_1,\ldots,f_c\in R$. For $T\subseteq [c]$, let $f_T=\prod_{i\in T} f_i$, and write $f=f_{[c]}$ for convenience. For $a\geq 1$, recall that we denote $\f^{[a]}:=f_1^a,\ldots,f_c^a$ regardless of whether $a$ is a power of the characteristic, and the notation $\f^{[a]}_T$ denotes the subsequence of $\f^{[a]}$ indexed by $T\subseteq [c]$.

Since $\f$ is permutable, we can use Proposition \ref{fedemb} to obtain identifications
\begin{equation}\label{anncoh}
  H^{c-i}_{\f_T}(R/\f_{\tT}) = (0:_{H^{c}_{\f}(R)}\f_{\tT})
  \end{equation}
for each subset $T\subseteq [c]$. Let $M=H^c_{\f}(R)$, and observe that there are inclusions $\iota_{T,S}:(0:_M \f_{\tS}) \hookrightarrow (0:_M \f_{\tT})$ whenever $S \subseteq T$.
\begin{definition}
  Let $R$ be a Noetherian ring and let $\f$ be a permutable regular sequence of codimension $c\geq 1$. Let $M=H^c_{\f}(R)$. The \textit{$\DDelta$ complex of $\f$}, denoted $\DDelta_{\fR}^\bullet(R)$, is the chain complex $(M^\bullet, \partial^\bullet)$ defined as follows.
  \begin{itemize}
  \item $M^i = \bigoplus_{|S| = i} (0:_M \f_{\tS})$ for $0\leq i\leq c$
  \item $\partial^i = \sum_{j=1}^c (-1)^j d^i_j$ where $d^i_j|_{(0:_M\f_{\tS})}$ is the direct sum of the inclusion maps $\iota_{S,T}:(0:_M\f_{\tS})\hookrightarrow (0:_M\f_{\tT})$ ranging over the sets $T\supseteq S$ of size $|T|=i+1$ such that $T\setminus S$ is the $j$th element of $T$, enumerated so that $t_a < t_b$ (as elements of $[c]$) when $a < b$.
  \end{itemize}
\end{definition}
The choice of differentials gives $\{M^i\}_{i = 0}^c$ the structure of a {\em semi-cosimplicial} $R$-module \cite[Def. 8.1.9, Ex. 8.1.6]{WeibelHomological}, which is to say, $d^{i+1}_k d^i_j = d^{i+1}_jd^i_{k-1}$ for $j < k$. Checking that $\partial^{i+1}\partial^i = 0$ is similar to checking that the chain maps in a \v{C}ech complex square to zero, and depends only on the semi-cosimplicial structure \cite[Def. 8.2.1]{WeibelHomological}. 

The complex $\DDelta_{\fR}^\bullet(R)$ comes equipped with quotient complexes induced by our enumeration $f_1, \dots, f_c$ of this regular sequence, meaning that if $\sigma: [c] \to [c]$ is a non-trivial bijection, then setting $g_i = f_{\sigma(i)}$, $1 \le i \le c$, the regular sequence $g_1, \dots, g_c$ induces a {\em distinct} collection of quotients.

\begin{definition}\label{filtration}
  Let $R$ be a Noetherian ring and let $\f$ be a permutable regular sequence of codimension $c\geq 1$. For fixed $n$ such that $1\leq n\leq c$, define the complex $\DDelta^\bullet_{\f}(R)_n$ by
  \[    
  \DDelta_{\fR}^i(R)_n := \oplus_{S \subset [n],\hspace{0.3em} |S|=i\,} (0 :_M \f_{\tS})\subseteq \DDelta_{\fR}^i(R)
  \]
with differentials $\partial^i_0 = \sum_{j=1}^n (-1)^j d^i_{j,n}$ defined so that the map $d^i_{j,n}|_{(0:_M\f_{\tS})}$ for $S\subseteq [n]$ is the direct sum of inclusions $\iota_{S,T}:(0:_M\f_{\tS})\hookrightarrow (0:_M\f_{\tT})$ ranging only over the sets $T$ of size $|T|=i+1$ such that $S\subseteq T\subseteq [n]$ and such that $T\setminus S$ is the $j$th element of $T$, enumerated so that $t_i < t_j$ (as elements of $[c]$) when $i < j$.
\end{definition}

For example, $\DDelta^\bullet_{\fR}(R)_n = (0 \to R/{\fR} \to 0)$ and $\DDelta^\bullet_{\fR}(A)_c = \DDelta_{\fR}(R)$. As a warning, $\DDelta_{\fR}^i(A)_n$ are \textit{not} subcomplexes of $\DDelta_{\fR}^\bullet(R)$. They are, however, quotient complexes. Let $K^\bullet_n$ denote the kernel of the surjection $\DDelta^\bullet_{\f}(R)_n \to \DDelta^\bullet_{\f}(R)_{n-1}$ defined term-by-term in the obvious way. 
  \[
  \begin{tikzcd}
    0 \ar[r]& K^\bullet_n \ar[r]& \DDelta^\bullet_{\f}(R)_n \ar[r]& \DDelta^\bullet_{\f}(R)_{n-1} \ar[r] & 0
  \end{tikzcd}
  \]
  For our calculation of the cohomology of $\DDelta^\bullet_{\fR}(R)$, the key observations are as follows.
  \begin{proposition}\label{quotcomp}
    Let $R$ be a Noetherian ring and let $\f$ be a permutable regular sequence of codimension $c\geq 1$. For fixed $n$ such that $1\leq n\leq c$, where all set complements (e.g. $\widetilde{[n]}$) are taken within $[c]$, we have the following.
    \begin{enumerate}
    \item $\DDelta^\bullet_{\f}(R)_n = \DDelta^\bullet_{\f_{[n]}}(R/\f_{\widetilde{[n]}})$.
    \item $K_n^\bullet = H^1_{f_n}(\DDelta^\bullet_{\f_{[n-1]}}(R/\f_{\widetilde{[n]}}))[-1]$.
    \end{enumerate}
    where $[-1]$ denotes the right-shift operator on cohomologically indexed complexes.
  \end{proposition}
  \begin{proof}
    Let $M=H^c_{\f}(R)$. The module $\DDelta^i_{\f}(R)_n$ is the direct sum of annihilators $(0:_M \f_{\tS})$ ranging over all subsets $S\subseteq [n]$ of size $|S|=i$. In particular, we have $[c]-[n]\subseteq [c]-S$ for all such $S$, so that
    \[
    (0:_M \f_{[c]-S}) = (0:_{(0:_M \f_{[c]-[n]})} \f_{[n]-S}) = (0:_{M_n} \f_{[n]-S})
    \]
    where $M_n=H^n_{\f_{[n]}}(R/{\f_{[c]-[n]}})$, and thus, $\DDelta^i_{\f}(R)_n=\DDelta^i_{\f_{[n]}}(R/\f_{[c]-[n]})$. The agreement of the differentials in the complexes $\DDelta^\bullet_{\f}(R)_n$ and $\DDelta^\bullet_{\f_{[n]}}(R/\f_{[c]-[n]})$ is a straightforward consequence of their definitions.

    Concerning $K^i_n$, the kernel of $\DDelta^\bullet_{\f}(R)_n\twoheadrightarrow\DDelta^\bullet_{\f}(R)_{n-1}$ is the direct sum of the annihilators $(0:_M \f_{\tS})$ ranging over subsets $S\subseteq [n]$ of size $|S|=i$ such that $n\in S$. We therefore have an isomorphism
    \[
    (0:_M \f_{[c]-S})= H^i_{\f_S}(R/\f_{[c]-S})=H^1_{f_n}\left(H^{i-1}_{\f_{S-\{n\}}}(R/\f_{[c]-S})\right)=H^1_{f_n}\left(\left(0:_{M_n}\f_{[n-1]-(S-\{n\})}\right)\right)
    \]
    where, once again, $M_n=H^n_{\f_{[n]}}(R/\f_{[c]-[n]})$. The sets $S-\{n\}$ for $S\subseteq [n]$ of size $|S|=i$ such that $n\in S$ correspond precisely to the subsets $S^\prime\subseteq [n-1]$ of size $|S^\prime|=i-1$. Thus, $K^i_n=\DDelta^{i-1}_{\f_{[n-1]}}(R/\f_{[c]-[n]})$. Confirming agreement of the corresponding differentials is straightforward.
  \end{proof}

\begin{example}\label{filtration and quotients}
  Suppose $c = 4$. When $n = 0$, $\DDelta_{\fR}^i(R)_n = (0:_M {\fR}) \cong R/{\fR}$ for $i = 0$ and $\DDelta_{\fR}^i(R)_n = 0$ for $i > 0$. We show $\DDelta_n$ for $1\leq n \leq 4$, identifying $(0:_M \f_{\tS})$ with $H^{i}_{\f_S}(R/\f_{\tS})$. Components of $\partial^i$ corresponding to $d^i_1$, $d^i_2$, $d^i_3$, and $d^i_4$ are indicated in red, blue, dashed red, and dashed blue, with (dashed or solid) red indicating a sign change.
  
  \vskip2ex
  
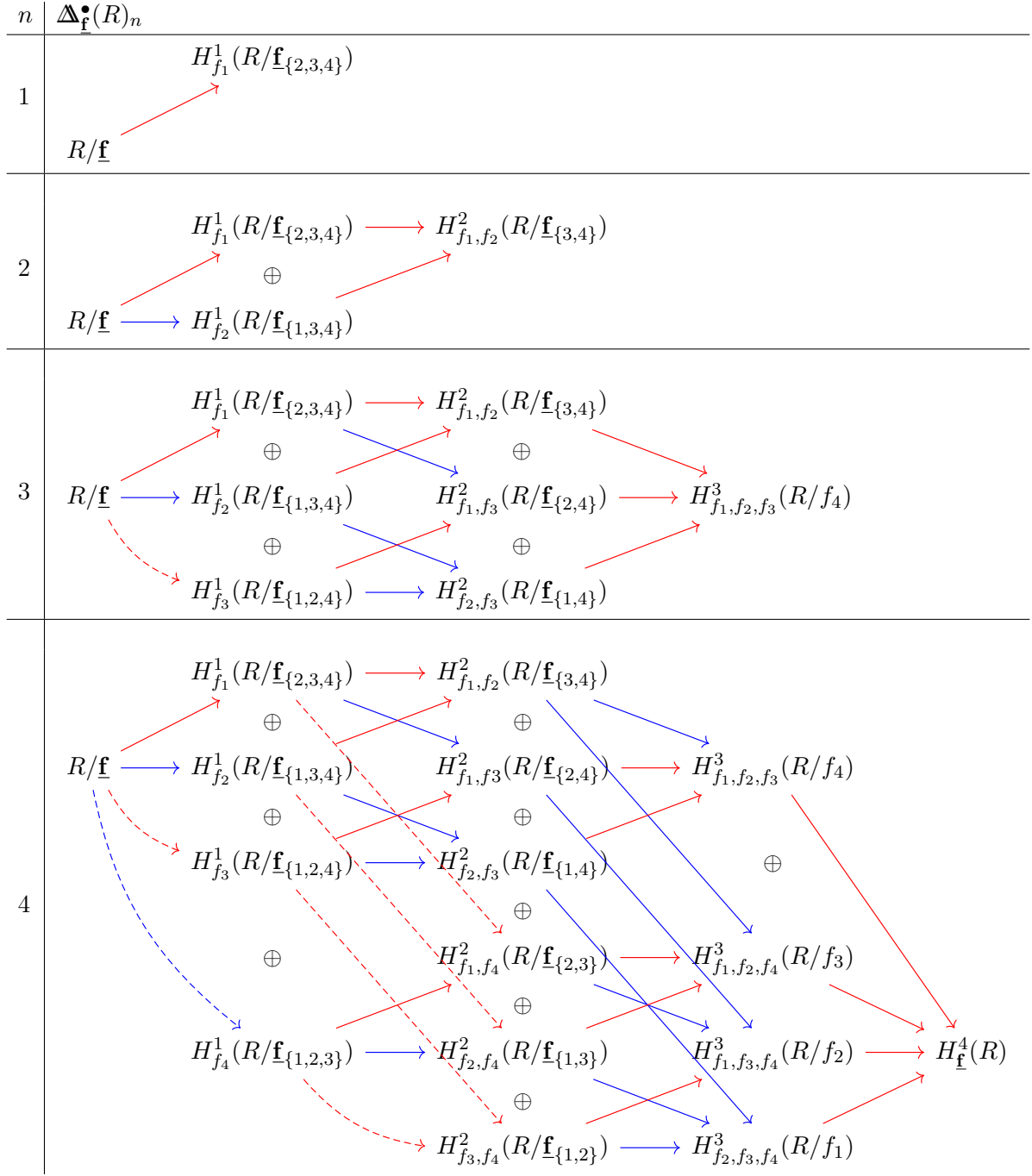
  \begin{tabular}{l | l}
    $n$ & $\DDelta_{\fR}^\bullet(R)_n$ \\
    \hline
    1 & 
    \begin{tikzcd}
      & H^1_{f_1}(R/\f_{\{2,3,4\}})\\
      R/{\fR} \ar[ru, red] &  & 
    \end{tikzcd}
    \\
    \hline
    \\
    2 & 
    \begin{tikzcd}
                                     & H^1_{f_1}(R/\f_{\{2,3,4\}}) \ar[r, red]\ar[d,phantom,"\oplus"] & H^2_{f_1,f_2}(R/\f_{\{3,4\}}) &\\
       R/{\fR} \ar[r, blue]\ar[ru, red] & H^1_{f_2}(R/\f_{\{1,3,4\}}) \ar[ru, red]                  &                               & 
    \end{tikzcd}
    \\
    \hline
    \\
    3 & 
    \begin{tikzcd}
      & H^1_{f_1}(R/\f_{\{2,3,4\}}) \ar[r, red]\ar[rd, blue]\ar[d,phantom,"\oplus"] & H^2_{f_1, f_2}(R/\f_{\{3,4\}}) \ar[rd, red] \ar[d, phantom, "\oplus"]
      &   & 
      \\
       R/{\fR} \ar[rd, dashed, red, bend right = 20]\ar[ru, red]\ar[r, blue]& H^1_{f_2}(R/\f_{\{1,3,4\}}) \ar[rd, blue]\ar[ru, red]\ar[d,phantom,"\oplus"] & H^2_{f_1,f_3}(R/\f_{\{2,4\}}) \ar[r, red] \ar[d,phantom,"\oplus"]& H^3_{f_1,f_2,f_3}(R/f_4)&
      \\
      & H^1_{f_3}(R/\f_{\{1,2,4\}}) \ar[ru, red]\ar[r, blue] & H^2_{f_2, f_3}(R/\f_{\{1,4\}}) \ar[ru, red] & & 
    \end{tikzcd}
    \\
    \hline
    \\
    4 &
    \begin{tikzcd}
      & H^1_{f_1}(R/\f_{\{2,3,4\}})\ar[rddd,red,dashed] \ar[r, red]\ar[rd, blue]\ar[d,phantom,"\oplus"] & H^2_{f_1, f_2}(R/\f_{\{3,4\}})\ar[rddd,blue] \ar[rd, blue] \ar[d, phantom, "\oplus"]
      &  & 
      \\
       R/{\fR} \ar[rd, dashed, red, bend right = 20]\ar[ru, red]\ar[r, blue]\ar[rddd,blue,dashed, bend right = 20]& H^1_{f_2}(R/\f_{\{1,3,4\}})\ar[rddd,red,dashed] \ar[rd, blue]\ar[ru, red]\ar[d,phantom,"\oplus"] & H^2_{f_1,f3}(R/\f_{\{2,4\}})\ar[rddd,blue] \ar[r, red] \ar[d,phantom,"\oplus"]& H^3_{f_1,f_2,f_3}(R/f_4)\ar[dd,phantom,"\oplus"]\ar[rddd,red]&
      \\
      & H^1_{f_3}(R/\f_{\{1,2,4\}})\ar[dd,phantom,"\oplus"]\ar[rddd,red,dashed] \ar[ru, red]\ar[r, blue] & H^2_{f_2, f_3}(R/\f_{\{1,4\}})\ar[rddd,blue]\ar[d,phantom,"\oplus"] \ar[ru, red] & &
      \\
      & & H^2_{f_1, f_4}(R/\f_{\{2,3\}})\ar[d,phantom,"\oplus"]\ar[r,red]\ar[dr,blue]  & H^3_{f_1,f_2,f_4}(R/f_3)\ar[dr,red]& 
      \\
      & H^1_{f_4}(R/\f_{\{1,2,3\}}) \ar[ru, red]\ar[r, blue]\ar[rd,red,dashed,bend right=20] & H^2_{f_2, f_4}(R/\f_{\{1,3\}})\ar[d,phantom,"\oplus"]\ar[ur,red]\ar[dr,blue]  & H^3_{f_1,f_3,f_4}(R/f_2)\ar[r,red]& H^4_{\f}(R)
      \\
      & & H^2_{f_3, f_4}(R/\f_{\{1,2\}})\ar[ur,red]\ar[r,blue]  & H^3_{f_2,f_3,f_4}(R/f_1)\ar[ru,red]& 
    \end{tikzcd}    
  \end{tabular}
  
  \vskip5ex
  
  The subcomplexes $K^\bullet_2$, $K^\bullet_3$, and $K^\bullet_4$ are displayed with terms generally to the lower right. For example, $K^\bullet_4$ consists of the terms in $\DDelta^\bullet_{\f}(R)_4$ that involve the local cohomology of quotients $R/\f_S$ for subsets $S\subseteq \{1,2,3\}$.
\end{example}

\section{The Cohomology of the $\DDelta$ Complex.}

We continue with the notation of the last section. Let $R$ be a Noetherian ring, let $\f$ be a permutable regular sequence of codimension $c\geq 1$, and let $f=\prod_{i=1}^c f_i$. We denote by $\DDelta^\bullet_{\fR}(R)^+$ the augmented chain complex equal to $\DDelta^\bullet_{\fR}(R)$ in degrees $\leq c$, with augmentation $\DDelta^{c+1}_{\fR}(R)^+ := H^c_{\fR}(R)$ and differential $\partial^c:\DDelta^c_{\f}(R)^+\to \DDelta^{c+1}_{\f}(R)^+$ given by
the multiplication by $f$ map $H^c_{\f}(R)\to H^c_{\f}(R)$.

\begin{lemma}\label{iterated local cohomology}
   Let $R$ be a Noetherian ring, let $\f$ be a permutable regular sequence of codimension $c\geq 1$. Suppose $H^i(\DDelta^\bullet_{\fR}(R)^+) = 0$ for all $1\leq i\leq c+1$, and let $h \in R$ extend $\f$ to a permutable regular sequence $f_1, \dots, f_c,h$ of codimension $c+1$. Then $H^i(H^1_{h}(\DDelta_{\fR}(R)^+)) = 0$ for all $1\leq i\leq c+1$.
\end{lemma}
\begin{proof}
  For the sake of notational convenience, write $\DDelta^\bullet = \DDelta_{\fR}^\bullet(R)^+$. We compute $H^1_{h}$ via the double complex $\check{C}(h;R)\otimes_R \DDelta^\bullet$, i.e. $0 \to \DDelta^\bullet \to (\DDelta^\bullet)_h \to 0$, as shown below.
  \[
  \begin{tikzcd}
    & 0 \ar[d]              &     0 \ar[d]                &                       &                   0 \ar[d]        &               0 \ar[d]               & \\
    0       \ar[r]& \DDelta^0 \ar[d]\ar[r]&     \DDelta^1   \ar[d]\ar[r]& \cdots \ar[r]   & \DDelta^c \ar[d]\ar[r, "f"]  & \DDelta^{c+1} \ar[d]\ar[r]& 0\\
    0       \ar[r]& (\DDelta^0)_h \ar[d]\ar[r]&     (\DDelta^1)_h   \ar[d]\ar[r]& \cdots \ar[r]   & (\DDelta^c)_h \ar[d]\ar[r, "f"]  & (\DDelta^{c+1})_h \ar[d]\ar[r]& 0\\
    & 0                     &     0                       &                       &                   0               &               0.                    & 
  \end{tikzcd}
  \]
  
  By hypothesis, $H^i(\DDelta^\bullet) = 0$ for all $i$, so $H^i((\DDelta^\bullet)_h) = 0$ for all $i$ as well. If we first take cohomology horizontally, we therefore obtain $E_1=E_\infty=0$. The cohomology of the totalization is therefore zero. If, on the other hand, we take vertical cohomology first, then we arrive at a double complex with one nonzero row, 
  \[
  \begin{tikzcd}
    0\ar[r]& H^1_{h}(\DDelta^0) \ar[r]& \cdots \ar[r]& H^1_{h}(\DDelta^{c+1})  \ar[r]& 0.
  \end{tikzcd}
  \]
  
The modules $H^i(H^1_{(h)}(\DDelta^\bullet))$ appear now as the horizontal cohomology, with $E_2=E_\infty$, and since only a single row is nonzero, they yield the cohomology of the totalization, which vanishes.
\end{proof}

We are now ready to prove the main theorem of this section.

\begin{theorem}\label{DDelta is exact}
  Let $R$ be a Noetherian ring, and let $\f$ be a permutable regular sequence of codimension $c\geq 1$. Then $H^i(\DDelta_{\fR}^\bullet(R)) = 0$ for $0 \le i < c$, and $H^c(\DDelta_{\fR}^\bullet(R)) \cong H^c_{\fR}(R)$, with augmentation map isomorphic to multiplication by $f := \prod_1^c f_i$. 
  
  If $R$ has prime characteristic $p > 0$, then $\DDelta_{\fR}^\bullet(R)$ is a complex of $R\langle F \rangle$-modules considered with their Fedder actions, and the induced Frobenius action on the augmentation $H^c_{\fR}(R) \cong H^c(\DDelta_{\fR}^\bullet(R))$ is the natural action.
\end{theorem}
\begin{proof}
  The differentials of $\DDelta^\bullet_{\fR}(R)$ are direct sums of inclusions of submodules of $M=H^c_{\f}(R)$, and by Proposition \ref{fedemb} these inclusions are Fedder-action linear in characteristic $p>0$. Our statement about the induced Frobenius action on $H^c(\DDelta^\bullet_{\fR}(R))$ is proven at the end of this argument; the bulk of this proof is calculation of the cohomology. 
    
  We proceed by by induction on $c$, with base case $c = 1$. With $f=f_1$, the complex $\DDelta^\bullet_{f}(R)$ is    
  \[
  \begin{tikzcd}
    0 \ar[r] & R/f \ar[rrr, "(r + fR) \mapsto \lbrak r/f \rbrak"]& & & H^1_{f}(R) \ar[r] & 0.
  \end{tikzcd}
  \]
  
  The map $R/f\to H^1_f(R)$ shown above is clearly injective, so $H^0(\DDelta^\bullet_f(R))=0$. Moreover, the image of $R/f\to H^1_f(R)$ is precisely the kernel of multiplication by $f$. Multiplication by $f$ on $H^1_f(R)$ is surjective, and the exactness of $0\to R/f\to H^1_f(R)\xrightarrow{f} H^1_f(R)\to 0$ implies that $H^1(\DDelta^\bullet_f(R))=H^1_f(R)$.
  
  Now assume the theorem has been proven for any permutable regular sequence of codimension $c \geq 1$ in a Noetherian ring. Let $\f,h=f_1, \dots, f_c, h \in R$ be a permutable regular sequence of codimension $c+1$. From Proposition \ref{quotcomp}, $\DDelta^\bullet_{{\fR}}(R/h)$ is the quotient complex $\DDelta^\bullet_{\f,h}(R)_c$ of $\DDelta^\bullet_{\fR,h}(R)=\DDelta^\bullet_{\fR}(R)_{c+1}$. The kernel $K^\bullet_{c+1}$ of this quotient is isomorphic to $H^1_{h}(\DDelta^\bullet_{{\fR}}(R))[-1]$, giving us the short exact sequence of complexes shown below.
  \begin{equation}\label{sescomp}
  \begin{tikzcd}
    0 \ar[r] & H^1_{h}(\DDelta^\bullet_{{\fR}}(R))[-1] \ar[r] & \DDelta^\bullet_{\fR,h}(R) \ar[r] & \DDelta^\bullet_{{\fR}}(R/h) \ar[r] & 0,
  \end{tikzcd}
  \end{equation}
  
  Applying \cref{iterated local cohomology} and the induction hypothesis, we see $H^i\left(H^1_{h}(\DDelta^\bullet_{{\fR}}(R))[-1]\right)=0$ for $i\leq c$, and that $H^{c+1}\left(H^1_{h}(\DDelta^\bullet_{{\fR}}(R))[-1]\right) = H^{c+1}(\DDelta_{{\fR,h}}(R)) = H^c_{{\fR}}(R)$. Likewise, we may apply the induction hypothesis to the $\DDelta^\bullet$ complex of the regular sequence $\f$ of codimension $c$ in $R/h$ to obtain $H^i(\DDelta^\bullet_{{\fR}}(R/h)) = 0$ for $i<c$ and $H^c(\DDelta^\bullet_{{\fR}}(R/h)) = H^c_{{\fR}}(R/h)$.
  
  We now study the long exact sequence in cohomology from the short exact sequence \eqref{sescomp}. To simplify notation, let $\DDelta^\bullet := \DDelta^\bullet_{\fR,h}(R)$, $\DDelta^\bullet_1 := \DDelta^\bullet_{{\fR}}(R/h)$, and $K^\bullet := H^1_{h}(\DDelta^\bullet_{{\fR}}(R))[-1]$. We obtain
  \[
  \begin{tikzcd}
    \cdots \ar[r] & H^i(K^\bullet) \ar[r]& H^i(\DDelta^\bullet) \ar[r]& H^i(\DDelta^\bullet_1) \ar[r, "\delta"]& H^{i+1}(K^\bullet) \ar[r] & \cdots
  \end{tikzcd}
  \]  
  We immediately see that $H^i(\DDelta^\bullet) = 0$ for $i < c$. Using that $H^c(K^\bullet) = 0$, we have an exact sequence  
  \begin{equation}\label{delta sequence}
    \begin{tikzcd}
      0 \ar[r]& H^c(\DDelta^\bullet) \ar[r]& H^c(\DDelta^\bullet_1) \ar[r, "\delta"]& H^{c+1}(K^\bullet) \ar[r] & H^{c+1}(\DDelta^\bullet) \ar[r]& 0.
    \end{tikzcd}
  \end{equation}
  
  We claim that $\delta$ is injective. To see this, we start with recalling the construction of $\delta$. We begin with the map from row $c$ to row $c+1$ in the short exact sequence of complexes.
\[
  \begin{tikzcd}
    0 \ar[r]& K^c\ar[r]\ar[d,"\partial^c_K"]& \DDelta^c \ar[r]\ar[d,"\partial^c_{\DDelta}"]& \DDelta_1^c\ar[r]\ar[d,"\partial^c_{\DDelta_1}"]& 0\\
    0 \ar[r]& K^{c+1}\ar[r]& \DDelta^{c+1} \ar[r]& 0\ar[r]& 0
  \end{tikzcd}
  \]
Let $M=H^{c+1}_{\f,h}(R)$. We identify $\DDelta^c_1=(0:_M h)$ and $K^c=\bigoplus_{i=1}^c (0:_M f_i)$. Let $\iota:\DDelta^c_1\hookrightarrow \DDelta^c$ denote the obvious splitting.
\[
  \begin{tikzcd}
    0 \ar[r]& \bigoplus_{i=1}^c (0:_M f_i)\ar[r]\ar[d,"\partial^c_K"]& (0:_Mh) \oplus\left(\bigoplus_{i=1}^c (0:_M f_i)\right) \ar[r,shift right=1]\ar[d,"\partial^c_{\DDelta}"]& (0:_M h)\ar[r]\ar[l,"\iota"',shift right = 1]\ar[d,"\partial^c_{\DDelta_1}"]& 0\\
    0 \ar[r]& M\ar[r]& M \ar[r]& 0\ar[r]& 0
  \end{tikzcd}
  \]
  A class $\lbrak \eta \rbrak_{\DDelta_1} \in H^c(\DDelta^\bullet_1)$ is represented by $\eta \in (0:_M h)$, where we write a subscript $\lbrak \cdots \rbrak_C$ to indicate the complex $C^\bullet$ with respect to which we're taking cohomology. By definition,  
  \[\delta(\lbrak \eta \rbrak_{\DDelta_1}) = \lbrak \partial^c_{\DDelta}(\iota(\eta)) \rbrak_K \in H^{c+1}(K).\]  
  If $\delta(\lbrak \eta \rbrak_{\DDelta_1}) = 0$, then $\partial^c_{\DDelta}(\iota(\eta))$ is in the image of $\partial_K^c$, 
  which is $\sum_{j=1}^c (0:_M f_j)$. Note that $\partial^c_{\DDelta}\circ \iota$ is just the inclusion map $(0:_M h)\hookrightarrow M$, possibly up to a sign change, so to say that $\partial^c_{\DDelta}(\iota(\eta))\in \sum_{j=1}^c (0:_M f_j)$ means precisely that
  \[ \eta \in (0:_M h) \cap \left(\sum_{j=1}^c (0:_M f_j) \right)=\text{Im}(\partial^{c-1}_{\DDelta_1}), \]  
  Thus, $\lbrak \eta \rbrak_{\DDelta_1} = 0$, and $\delta$ is injective.
  
  We conclude that $H^c(\DDelta^\bullet) = 0$ from the sequence \eqref{delta sequence}. Following the reasoning of the last paragraph, the image of $\partial^c_{\DDelta}$ is $(0:_M h) + \sum_{j=1}^c (0:_M f_j)$, which is the kernel of multiplication by $hf = h \left(\prod_1^c f_j\right)$. Thus, the augmentation map $\DDelta^{c+1}\twoheadrightarrow H^{c+1}(\DDelta^\bullet)$ is (by definition) the quotient $M\twoheadrightarrow M/(0:_M fh)$. Since multiplication by $fh$ is surjective on $M$, this provides an isomorphism $\varphi: H^{c+1}(\DDelta^\bullet)\to M$,
  \[
  \begin{tikzcd}
  0\ar[r]& \text{Im}(\partial^c_{\DDelta})\ar[d,equals]\ar[r] & \DDelta^{c+1}\ar[d,equals]\ar[r,"\text{aug.}"]& H^{c+1}(\DDelta^\bullet)\ar[d,"\varphi"]\ar[r]& 0\\
  0\ar[r]& (0:_M fh)\ar[r]& M\ar[r,"fh"]& M\ar[r]& 0
  \end{tikzcd}
  \]
  So, under the isomorphism $\varphi$ that identifies $H^{c+1}(\DDelta^\bullet)$ with $M=H^{c+1}_{\f,h}(R)$, the augmentation map $\DDelta^{c+1}\xrightarrow{\text{aug.}}H^{c+1}(\DDelta^\bullet)$ is isomorphic to the multiplication map $M\xrightarrow{fh} M$, as claimed. Except for the final statement about the induced Frobenius action in characteristic $p > 0$, we have proven the theorem. 
    
  The final statement comes down to a direct calculation. For a given representative $\eta\in \DDelta^{c+1}=H^{c+1}_{\f,h}(R)_{\text{fed}}$, the induced Frobenius action $\overline{F}$ on $H^{c+1}(\DDelta^\bullet)$ sends $\lbrak \eta\rbrak_{\DDelta}$ to $\lbrak F_{\text{fed}}(\eta)\rbrak_{\DDelta}$. Under the identification $\varphi: H^{c+1}(\DDelta^\bullet)\xrightarrow{\sim} H^{c+1}_{\f,h}(R)$, the augmentation map $\eta\mapsto \lbrak \eta\rbrak_{\DDelta}$ is the multiplication $\eta\mapsto (fh) \eta$. Thus,
  \[
  \overline{F}((fh)\eta) = (fh)F_{\text{fed}}(\eta) = (fh)^p F_{\text{nat}}(\eta) = F_{\text{nat}}((fh)\eta)
  \]
so $\overline{F} = F_{\text{nat}}$, as desired.
\end{proof}

\section{An Application to Closed Support}
Given a Noetherian ring $S$ and an ideal $I$, the local cohomology modules $H^i_I(S)$ for $i\leq \text{ht}(I)+1$ are, in a sense, fully general among all local cohomology modules of the form $H^\alpha_{\mathfrak{a}}(S)$. Given and ideal $I$ and an index $i\geq 0$, by \cite[Theorem 4.1]{LewisLocalCohomologyParameter}, there exists some possibly larger ideal $I^\prime\supseteq I$ such that $i\leq \text{ht}(I^\prime)+1$ with the functor $H^i_{I^\prime}(-)$ naturally isomorphic to $H^i_I(-)$. In the case where $S$ is Cohen-Macaulay (cf. \cite[Theorem 3]{HellusLocalCohomology} when $S$ is local) our attention is therefore restricted to the cases $H^{\text{ht}(I)}_I(S)$ and $H^{\text{ht}(I)+1}_I(S)$. It is easy to show that $H^{\text{ht}(I)}_I(S)$ has a finite set of associated primes, so questions about the support or associated primes of the local cohomology of a Cohen-Macaulay ring can, without loss of generality, be posed for modules of the form $H^{\text{ht}(I)+1}_I(S)$.

It remains an open question \cite[Question 2]{hochreview} whether the local cohomology of a complete intersection ring must have closed support. Given a regular ring $R$, a regular sequence $\f=f_1,\ldots,f_c$, and an ideal $I$ containing ${\fR}$, the closed support problem for a local complete intersection ring is, by the preceding discussion, determined by the behavior of modules of the form $H^{\text{ht}(I/{\fR})+1}_{I/{\fR}}(R/{\fR})$. Indeed, the hypersurface support theorems of Hochster and N\'{u}\~{n}ez-Betancourt \cite[Corollary 4.13]{HochsterNunezBetancourt} or Katzman and Zhang \cite[Theorem 7.1]{KatzmanZhang} may be interpreted as the statement that in prime characteristic $p>0$, all modules of the form $H^{\text{ht}(I/f)+1}_{I/f}(R/f)$ have closed support. 

In this section, we will show that for a permutable regular sequence $\f=f_1,\ldots,f_c$ in a regular ring $R$ of prime characteristic $p>0$, the module $H^{\text{ht}(I/{\fR})+c}_{I/{\fR}}(R/{\fR})$ has closed support for any ideal $I$ satisfying the vanishing hypothesis $H^i_I(R)=0$ for $\text{ht}(I)<i<\text{ht}(I)+c$.

\begin{lemma}
Let $R$ be a Noetherian ring of prime characteristic $p>0$, and let $M$ be a finitely generated $R\ncfrob$-module. The support of $M$ is closed.
\end{lemma}
\begin{proof}
See \cite[Lemma 4.5]{HochsterNunezBetancourt}.
\end{proof}

An $R\ncfrob$-linear homomorphic image of a finitely generated $R\ncfrob$ module is still finitely generated, and thus, still has closed support. However, we caution that $R\ncfrob$ is not left-Noetherian, so $R\ncfrob$-submodules of a finitely generated $R\ncfrob$-module need not be finitely generated.

\begin{lemma}
  Let $R$ be a regular ring of prime characteristic $p>0$, let $\f=f_1,\ldots,f_c$ be a regular sequence, and let $I$ be an ideal containing ${\fR}$. Denote by $H^c_{\fR}(R)_{\text{nat}}$ the $R\ncfrob$-module $H^c_{\fR}(R)$ equipped with its natural action. Then for any $i\geq 0$, the $R\ncfrob$-module $H^i_I(H^c_{\fR}(R)_{\text{nat}})$ (with the Frobenius action induced from $H^c_{\fR}(R)_{\text{nat}}$) is finitely generated over $R\ncfrob$. Moreover, the module $H^i_I(H^c_{\fR}(R))$ has a finite set of associated primes.
\end{lemma}
\begin{proof}
  Since $H^j_{\fR}(R)=0$ unless $j=c$, we may identify $H^i_I(H^c_{\fR}(R))$ with $H^{i+c}_I(R)$. The $R\ncfrob$-module $H^i_I(H^c_{\fR}(R)_{\text{nat}})$ with action induced by $H^c_{\fR}(R)_{\text{nat}}$ is isomorphic to the $R\ncfrob$-module $H^{i+c}_I(R)_{\text{nat}}$. This is finitely generated over $R\ncfrob$ by \cite[Lemma 4.8]{HochsterNunezBetancourt}, and the statement about associated primes follows from \cite{hush} or \cite{lyufmod}.
\end{proof}

Assume from this point onward that the regular sequence $\f=f_1,\ldots,f_c\in R$ is permutable, and let $f=\prod_{i=1}^c f_i$. Given a subset $T\subseteq [c]$, recall that we use $\f_T$ to denote the subsequence of $\f$ indexed by $T$, and let $f_T=\prod_{i\in T}f_i$. For any such subset $T$ of size $|T|=b$, let $N^a_{\f_T}$ denote the kernel of the $i$th differential $\partial^a:\DDelta^a_{\f_T}(R)\to\DDelta^{a+1}_{\f_T}(R)$ when $a<b$, and let $N^b_{\f_T}$ denote the kernel of the augmentation map $\DDelta^b_{\f_T}(R)\to H^b(\DDelta^b_{\f_T}(R))$. By Theorem \ref{DDelta is exact}, we have the following:
\begin{itemize}
\item $N^1_{\f_T}=R/\f_T$.
\item $N^b_{\f_T}$ fits into an exact sequence \[0\to N^b_{\f_T}\to H^b_{\f_T}(R)_{\text{fed}}\xrightarrow{f_T}H^b_{\f_T}(R)_{\text{nat}}\to 0\]
\item For all values $1\leq a<b$, there is an exact sequence of the form
  \[0\to N^a_{\f_T}\to \bigoplus_{S\subseteq T,\,\, |S|={b-a}} H^a_{\f_{T-S}}(R/\f_S)\to N^{a+1}_{\f_T}\to 0.\]
  
\end{itemize}
  The compatibility of the differentials $\partial^i$ with the Fedder actions of each term in the $\DDelta^\bullet_{\f_T}(R)$ complex implies that each module of the form $N^a_{\f_T}$ carries an induced Frobenius action. The short exact sequences displayed above may therefore be understood over $R\ncfrob$, and the long exact sequences, including connecting homomorphisms, that result from applying a functor $\Gamma_I(-)$ for $I\supseteq {\fR}$ may also be understood over $R\ncfrob$, using the induced actions.
  
Note that the vanishing hypotheses of the following theorem are automatically satisfied if $R/I$ is Cohen-Macaulay (see \cite{PeskineSzpiroDimensionProjective}) or if $c=1$.

\begin{theorem}
  Let $R$ be a regular ring of prime characteristic $p>0$, let $\f=f_1,\ldots,f_c$ be a permutable regular sequence of codimension $c\geq 1$, and let $I\supseteq \f$ be an ideal such that $H^i_I(R)=0$ for $\text{ht}(I)<i<\text{ht}(I)+c$. The module $H^{\text{ht}(I/{\fR})+c}_{I/{\fR}}(R/{\fR})$ has Zariski closed support.
\end{theorem}
\begin{proof}
  For convenience, write $t=\text{ht}(I/{\fR})=\text{ht}(I)-c$. We will make heavy use of the $N^a_{\f_T}$ notation introduced in the preceding discussion.
  
  Our first aim is to show that the following three statements hold for all $b$ such that $1\leq b\leq c$, 

  \begin{enumerate}
  \item[(i)] For any subset $T\subseteq [c]$ of size $|T|=b$, and for all $a$ satisfying $\text{max}(1,3-c+b)\leq a\leq b$ (an empty range of values if $c\leq 2$), it holds that $H^{j}_I(N^a_{\f_T})=0$ whenever
        \[t+c+2-a\leq j\leq t+2c-1-b.\]    
  \item[(ii)] For any subset $T\subseteq [c]$ of size $|T|=b$, and for all $a$ satisfying $\text{max}(1,2-c+b)\leq a\leq b$ (an empty range if $c=1$) the module $H^{t+c+1-a}_I(N^a_{\f_T})$ is finitely generated over $R\ncfrob$.
  \item[(iii)] For any subset $T\subseteq [c]$ of size $|T|=b$, and for all $\text{max}(1,2-c+b)\leq a\leq b$ (an empty range of values if $c=1$), the module $H^{t+2c-b}_I(N^a_{\f_T})$ has a finite set of associated primes.
  \end{enumerate}

The proof is by induction on $b$, beginning with the case $b=1$. We will actually start by showing that the statements hold whenever $b=a$, i.e., for the modules $N^b_{\f_T}$ when $|T|=b$. This immediately implies the $b=1$ case, since $N^1_{f_j}=R/f_j$. So, fix $1\leq b\leq c$ and let $T\subseteq [c]$ be a subset of size $|T|=b$. Concerning the module $N^b_{\f_T}$, we have an exact sequence
  \[0\to N^b_{\f_T}\to H^b_{\f_T}(R)_{\text{fed}}\xrightarrow{f_T} H^b_{\f_T}(R)_{\text{nat}} \to 0\]
  From the long exact sequence induced by $\Gamma_I(-)$ along with our vanishing hypothesis $H^i_I(R)=0$ for $t+c+1\leq i\leq t+2c-1$, it is readily verified that (i) so long as $c\geq 3$, $H^j_I(N^b_{\f_T})=0$ for $t+c+2-b\leq j\leq t+2c-1-b$, (ii) that so long as $c\geq 2$, $H^{t+c+1-b}_I(N^b_{\f_T})$ is a homomorphic image of $H^{t+c}_I(R)_{\text{nat}}$ (and is therefore finitely generated over $R\ncfrob$), and finally, (iii) that so long as $c\geq 2$, $H^{t+2c-b}_I(R/\f_T)$ is isomorphic to a submodule of $H^{t+2c}_I(R)$ (and hence has a finite set of associated primes).

  Now take $b$ in the range $2\leq b\leq c$, since the case $b=1$ is proven. Suppose that the statements (i)--(iii) have been shown for subsets $S\subseteq [c]$ of size $|S|<b$, and fix $T\subseteq [c]$ any subset of size $b$. For this set $T$, we will demonstrate the claims (i)--(iii) about the modules $N^a_{\f_T}$ by a decreasing induction on $a$. The case $a=b$ has already been shown.
  
  Fix $a< b$ and suppose we've proven (i)--(iii) for the modules $N^{r}_{\f_T}$ whenever $a<r\leq b$. We will show that the statements hold for the module $N^a_{\f_T}$ using the short exact sequence
  \[0\to N^a_{\f_T}\to \bigoplus_{S\subseteq T,\,\, |S|={b-a}} H^a_{\f_{T-S}}(R/\f_S)\to N^{a+1}_{\f_T}\to 0,\]

  To show claim (i), fix $j$ in the range $t+c+2-a\leq j\leq t+2c-1-b$ and consider the exact sequence
  \[\cdots\to H^{j-1}_I(N^{a+1}_{\f_T})\to H^j_I(N^a_{\f_T})\to \bigoplus_{S\subseteq T,\,\, |S|={b-a}} H^{j}_{I}(H^a_{\f_{T-S}}(R/\f_S))\to \cdots\]
  note that for each subset $S\subseteq T$ of size $|S|=b-a$, we have \[H^{j}_I(H^a_{\f_{T-S}}(R/\f_S))= H^{j+a}_I(R/\f_S)\] where $R/\f_S=N^1_{\f_S}$. The inequality $t+c+1\leq j+a \leq t+2c-1-(b-a)$ gives us vanishing $H^{j+a}_{I}(N^1_{\f_S})=0$ for each subset $S\subseteq T$ of size $b-a$ by induction, since $|S|<|T|$. Since $t+c+2-(a+1) \leq j-1\leq t+2c-1-b$, we also have $H^{j-1}_I(N^{a+1}_{\f_T})=0$ by induction, since $a+1>a$. The vanishing of $H^j_I(N^a_{\f_T})$ follows at once.
  
  For claim (ii), the relevant exact sequence is
  \[\cdots\to H^{t+c-a}_I(N^{a+1}_{\f_T})\to H^{t+c+1-a}_I(N^a_{\f_T})\to \bigoplus_{S\subseteq T,\,\, |S|={b-a}} H^{t+c+1}_I(N^1_{\f_S})\to \cdots\]
Since $a\geq \text{max}(1,2-c+b)$, we have that $1\geq 3-c+(b-a)$, so claim (i) for the module $N^1_{\f_S}$ implies that $H^{t+c+1}_I(N^1_{\f_S})=0$. Thus, $H^{t+c+1-a}_I(N^a_{\f_T})$ is an $R\ncfrob$ homomorphic image of $H^{t+c+1-(a+1)}_I(N^{a+1}_{\f_T})$, which is finitely generated over $R\ncfrob$ by induction on $a$.

  To show claim (iii), consider the exact sequence
  \[\cdots\to H^{t+2c-b-1}_I(N^{a+1}_{\f_T})\to H^{t+2c-b}_I(N^a_{\f_T})\to \bigoplus_{S\subseteq T,\,\, |S|={b-a}} H^{t+2c-b+a}_I(N^1_S)\to \cdots\]
  The condition $a\geq 2-c+b$ implies that $a+1\geq 3-c+b$, so claim (i) for the module $N^{a+1}_{\f_T}$ shows that $H^{t+2c-1-b}_I(N^{a+1}_{\f_T})=0$. Thus $H^{t+2c-b}_I(N^a_{\f_T})$ is isomorphic to a submodule of a direct sum of modules of the form $H^{t+2c-(b-a)}_I(N^1_S)$, for $S\subseteq T$ a subset of size $b-a$. Since $2-c+(b-a)\leq 1$, each $H^{t+2c-(b-a)}_I(N^1_S)$ has a finite set of associated primes. The induction is complete and the claims (i)--(iii) have been demonstrated.
  
  We are now ready to show that $H^{t+c}_I(R/{\fR})=H^{t+c}_I(N^1_{\fR})$ has closed support. For $c\geq 2$ (the $c=1$ argument follows similarly, cf. Section 1), there is an exact sequence
  \[\cdots\to H^{t+c-1}_I(N^2_{\fR})\to H^{t+c}_I(N^1_{\fR})\to \bigoplus_{S\subseteq T,\,\, |S|=1} H^{t+c+1}_I(N^1_{\f_S})\to\cdots\]
  Since $2\geq \text{max}(1,2+c-c)$, the module $H^{t+c+1-2}_I(N^2_{\fR})$ is finitely generated over $R\ncfrob$, and thus, any $R\ncfrob$ homomorphic image of this module has closed support. Additionally, $1\geq 2-c+(c-1)$, so $H^{t+2c-(c-1)}_I(N^1_{\f_S})$ (for each singleton set $S\subseteq T$) has a finite set of associated primes. The claim about the support of $H^{t+c}_I(N^1_{\fR})$ follows at once.
\end{proof}

\subsection{Nesting of Supports}

In this subsection, we remark that the support of the local cohomology of a complete intersection cut out by a regular sequence $f_1,\ldots,f_c$ has a curious nesting property in relation to the supports of the local cohomologies of the complete intersections defined by subsequences of $f_1,\ldots,f_c$.
\newcommand{\h}{\mathbf{\underline{h}}}
\begin{theorem}
  Let $R$ be a Cohen-Macaulay ring of prime characteristic $p>0$, let $\f=f_1,\ldots,f_c$ be a permutable regular sequence, and let $I$ be an ideal containing ${\fR}$. For $T\subseteq [c]$, let $\f_T$ be the ideal generated by the subsequence of $f_1,\ldots,f_c$ indexed by $T$. For any $\delta\geq 0$,  
  \[\Supp H^{\text{ht}(I/\f_T)+\delta}_{I/\f_T}(R/\f_T)\subseteq \Supp H^{\text{ht}(I/{\fR})+\delta}_{I/{\fR}}(R/{\fR})\]  

  In particular, if $\h=f_1,\ldots,f_c,g_1,\ldots,g_t$ is a maximal length regular sequence in $I$ and if $\h$ is permutable, then
  \[\Supp H^{\text{ht}(I)+\delta}_I(R)\subseteq \Supp H^{\text{ht}(I/{\fR})+\delta}_{I/{\fR}}(R/{\fR})\subseteq \Supp H^\delta_{I/\h}(R/\h)\]
\end{theorem}
\begin{proof}
  Let $\h=f_1,\ldots,f_c,g_1,\ldots,g_t$ be a maximal length regular sequence contained in $I$. Via the obvious inclusions $T\subseteq [c]\subseteq [c+t]$, we may write ${\fR}=\h_{[c]}$ and $\f_T=\h_T$. Let $b=|T|$. Observe that
  \[H^{\text{ht}(I/\f_T)+\delta}_{I/\f_T}(R/\f_T)=H^{\delta}_I\left(H^{t+c-b}_{\h_{[c+t]-T}}(R/\h_T)\right),\]
  and that
  \[H^{\text{ht}(I/{\fR})+\delta}_{I/{\fR}}(R/{\fR})=H^{\delta}_I\left(H^{t}_{\h_{[c+t]-[c]}}(R/\h_{[c]})\right).\]
  Let $A=R/\h_T$, and consider the ring $R/\h_{[c]}$ as being cut out from $A$ by a regular sequence of length $c-b$ (indexed by the set $[c]-T$). The result follows by a straightforward induction using the following lemma.
\end{proof}

\begin{lemma}
  Let $A$ be a Noetherian ring, let $f_1,\cdots,f_t,h\in A$ be a permutable regular sequence, and let $\f=f_1,\dots,f_t$. Let $I$ be an ideal containing $\f,h$. Then for any $\delta\geq 0$,
  \[\Supp H^\delta_I(H^{t+1}_{\f,h}(R)) \subseteq \Supp H^\delta_I(H^t_{\f}(R/h))\]
\end{lemma}
\begin{proof}
  Suppose, after replacing $R$ by $R_P$ for some $P\in \Spec(R)$, we obtain $H^\delta_I(H^t_{\f}(R/h))=0$. We would like to show that $H^\delta_I(H^{t+1}_{\f,h}(R))=0$, where we recall that
  \[H^{t+1}_{\f,h}(R) = H^t_{\f}(H^1_{h}(R))= \varinjlim_n H^t_{\f}(R/h^n)\]
  It would therefore suffice to show that $H^\delta_I(H^t_{\f}(R/h^n))=0$ for all $n\geq 1$. By hypothesis, this is true when $n=1$, so fix $n>1$ and suppose for the sake of induction that $H^\delta_I(H^t_{\f}(R/h^j))=0$ for all $j<n$.
  
  Note that $(h^n :_R h) = h^{n-1}R$, i.e., the annihilator of $h$ in $R/h^nR$ is $h^{n-1}R/h^nR$, isomorphic as an $R$-module to $R/hR$. Mapping $R/hR$ onto the image of $h^{n-1}$, we get
  \[0\to R/h \xrightarrow{h^{n-1}} R/h^n\to R/h^{n-1}\to 0\]  
  inducing the exact sequence
  \[\cdots\to H^{t-1}_{\f}(R/h^{n-1})\to H^t_{\f}(R/h) \to H^t_{\f}(R/h^n)\to H^t_{\f}(R/h^{n-1})\to H^{t+1}_{\f}(R/h)\to\cdots\]

  The arithmetic rank of $\f$ is $t$, so $H^{t+1}_{\f}(R/h)=0$. Since $\h$ is permutable, $h^{n-1},f_1,\cdots,f_t$ is a regular sequence, and consequently, $f_1,\dots,f_t$ is an $R/h^{n-1}$-regular sequence. Since $\depth_{\f}(R/h^{n-1})=t$, we get $H^{t-1}_{\h}(R/h^{n-1})=0$. Thus, $0\to H^t_{\f}(R/h) \to H^t_{\f}(R/h^n)\to H^t_{\f}(R/h^{n-1})\to 0$ is exact, and so is  
  \[\cdots\to H^\delta_I(H^t_{\f}(R/h)) \to H^\delta_I(H^t_{\f}(R/h^n))\to H^\delta_I(H^t_{\f}(R/h^{n-1}))\to\cdots.\]
  
  We have $H^\delta_I(H^t_{\f}(R/h))=0$, and by induction $H^\delta_I(H^t_{\f}(R/h^{n-1}))=0$, so $H^\delta_I(H^t_{\h}(R/h^n))=0$.
\end{proof}

\printbibliography

\end{document}